\documentclass[11pt,a4paper,reqno]{amsart}
\usepackage[czech,english]{babel}
\usepackage{amscd,amssymb}
\usepackage{dsfont}
\usepackage{url}
\usepackage[all]{xy}
\usepackage{hyperref}
\title{Recollements from cotorsion pairs}

\thanks{Research supported by grant BIRD163492 ``Categorical homological methods in the study of algebraic structure''  and grant DOR1690814``Rings and categories of modules"}

\author[S. Bazzoni]{Silvana Bazzoni}
\address[Silvana Bazzoni]
{%
Dipartimento di Matematica Tullio Levi-Civita\\
Universit\`a di Padova \\
    Via Trieste 63, 35121 Padova (Italy)}
\email{bazzoni@math.unipd.it}
\author[M.Tarantino]{Marco Tarantino}
\address[Marco Tarantino]
{%
Dipartimento di Matematica Tullio Levi-Civita\\
Universit\`a di Padova \\
    Via Trieste 63, 35121 Padova (Italy)}
\email{marco.tarantino@math.unipd.it}

\renewcommand{\iff}{if and only if }

\newcommand{\wrt}{with respect to }
\newcommand{\sto}{ {\v{S}}{\v{t}}ov{\'{\i}}{\v{c}}ek}

\newcommand{\Sto}{\v{S}\v{t}ov\'{\i}\v{c}ek}

\newcommand{\bbZ}{\mathbb{Z}}

\newcommand{\Hom}{\operatorname{Hom}}
\newcommand{\ho}{\operatorname{H}}
\newcommand{\HOM}{\mathcal{H}\mathnormal{om}}

\newcommand{\Ext}{\operatorname{Ext}}

\newcommand{\A}{\mathcal{A}}
\newcommand{\B}{\mathcal{B}}
\newcommand{\C}{\mathcal{C}}
\newcommand{\D}{\mathcal{D}}
\newcommand{\E}{\mathcal{E}}
\newcommand{\F}{\mathcal{F}}
\newcommand{\G}{\mathcal{G}}

\newcommand{\I}{\mathcal{I}}

\newcommand{\K}{\mathcal{K}}
\newcommand{\clL}{\mathcal{L}}
\newcommand{\M}{\mathcal{M}}
\newcommand{\N}{\mathcal{N}}
\newcommand{\clP}{\mathcal{P}}
\newcommand{\clQ}{\mathcal{Q}}
\newcommand{\R}{\mathcal{R}}
\newcommand{\clS}{\mathcal{S}}

\newcommand{\T}{\mathcal{T}}

\newcommand{\V}{\mathcal{V}}
\newcommand{\X}{\mathcal{X}}
\newcommand{\Y}{\mathcal{Y}}
\newcommand{\W}{\mathcal{W}}
\newcommand{\Z}{\mathcal{Z}}
 \newcommand{\Ch}{\mathrm{Ch}}

\newcommand{\Modr}[1]{\mathrm{Mod}\textrm{-}{#1}}

\newcommand{\modr}[1]{\mathrm{mod}\textrm{-}{#1}}

\newcommand{\Flat}{\mathrm{Flat}}
\newcommand{\Cot}{\mathrm{Cot}}
\newcommand{\Inj}{\mathrm{Inj}}
\newcommand{\FpInj}{\mathrm{FpInj}}
\newcommand{\Proj}{\mathrm{Proj}}

\newcommand{\modR}{\mathrm{mod}\textrm{-}R}

\newcommand{\Add}{\mathrm{Add}}

\newcommand{\Prod}{\mathrm{Prod}}

\newcommand{\Ho}{\operatorname{Ho}}

\theoremstyle{plain}
\newtheorem{thm}{Theorem}[section]
\newtheorem{lem}[thm]{Lemma}
\newtheorem{prop}[thm]{Proposition}
\newtheorem{cor}[thm]{Corollary}

\newtheorem{fact}[thm]{Fact}

\theoremstyle{definition}
\newtheorem{defn}[thm]{Definition}

\newtheorem{nota}[thm]{Notation}
\newtheorem{expl}[thm]{Example}

\theoremstyle{remark}
\newtheorem{rem}[thm]{Remark}

\begin{document}

\begin{abstract} Given a complete hereditary cotorsion pair $(\mathcal{A},\mathcal{B})$ in a Grothendieck category $\mathcal{G}$,
the derived category $\D(\B)$ of the exact category $\B$ is defined as the quotient of the category $\Ch(\B)$, of unbounded complexes with terms in $\B$, modulo the subcategory $\widetilde{\B}$ consisting of the acyclic complexes with terms in $\B$ and cycles in $\B$.

We restrict our attention to the cotorsion pairs such that
$\widetilde{\mathcal{B}}$ coincides with the class $ex\B$ of the acyclic complexes of $\Ch(\G)$ with terms in $\B$. In this case the derived category $\D(\B)$
fits into a recollement

$
\xymatrix{\dfrac{ex\B}{\sim} \ar[r]^{inc} &{K(\B)} \ar@/^1pc/[l]\ar@/_1pc/[l] \ar[r]^{Q}
&{\dfrac{\Ch(\B)}{ex\B }}\ar@/^1pc/ [l] \ar@/_1pc/ [l] }
.$

We will explore the conditions under which $\mathrm{ex}\,\mathcal{B}=\widetilde{\mathcal{B}}$ and provide many examples.

Symmetrically, we prove analogous results for the exact category $\A$.

\end{abstract}

\maketitle

\setcounter{tocdepth}{1}
\tableofcontents

\section*{Introduction} The notion of cotorsion pairs goes back to the seventies when it was introduced by Salce~\cite{Sal} in the case of abelian groups.
It got an enormous impulse thanks to the discovery by Hovey~\cite{Hov07} of the bijective correspondence between abelian model structures and cotorsion pairs in abelian categories. Many examples of cotorsion pairs and the corresponding model structures have been illustrated by Gillespie \cite{G5} who also extended the notion to the case of exact categories.

A famous example of cotorsion pair is given by the pair $(\F, \C)$ where $\F$ is the class of flat objects. It gave rise to the celebrated Flat Cover Conjecture by Enochs and solved in \cite{BEE} in the case of module categories and in ~\cite{ElB} for Grothendieck categories. It is particularly important in categories with no nonzero projective objects like for instance the categories of coherent sheaves.

In \cite{Nee08} Neemann described the homotopy category of the projective objects as a localization of the homotopy category of flat objects and he obtained a recollement  with middle  term the homotopy category of flat modules. His recollement generalizes the classical one having the homotopy category of a ring $R$ as middle term, the derived category of $R$ as right term and the category of acyclic complexes modulo the homotopy relation as left term.

In this paper we exhibit many other examples of recollements of analogous type.

Our results are strongly based on the two papers ~\cite{G6} and \cite{G7} by Gillespie and also inspired by Becker's idea in \cite{Beck12} to consider triples of injective cotorsion pairs giving rise to model structures and to the corresponding recollements.

Starting from a complete hereditary cotorsion pair $(\A, \B)$ in a Grothendieck category, we consider triples of examples of injective and projective cotorsion pairs on the categories of unbounded complexes with components in the exact categories $\A$ or $\B$.
The examples are chosen in order that the associated model structures on the categories of complexes satisfy the assumptions allowing to build the relevant recollements.

Our aim is mainly to describe the homotopy categories $K(B)$ or $K(\A)$.

Imposing some mild conditions on a Grothendieck category $\G$ (which are always satisfied by module categories), Theorem~\ref{T:triple-B} and Theorem~\ref{T:triple-A} give recollements
\vskip0.7cm
\[
\xymatrix{\dfrac{ex\B}{\sim} \ar[rr]^{inc} &&{K(\B)} \ar@/^2pc/[ll]\ar@/_2pc/[ll] \ar[rr]^{Q}
&&{\dfrac{\Ch(\B)}{ex\B }}\ar@/^2pc/ [ll] \ar@/_2pc/ [ll] }
,\]
\vskip0.7cm

\[
\xymatrix{\dfrac{ex\A}{\sim} \ar[rr]^{inc} &&{K(\A)} \ar@/^2pc/[ll]\ar@/_2pc/[ll] \ar[rr]^{Q}
&&{\dfrac{\Ch(\A)}{ex\A }}\ar@/^2pc/ [ll] \ar@/_2pc/ [ll] }
\]
\vskip0.7cm
 where for every subcategory $\C$ of $\G$, $ex{\C}$ denotes the class of acyclic unbounded complexes with terms in $\C$.

 The first recollement  generalizes the recollement obtained by Krause (~\cite{Kr05}) where the middle term is the homotopy category of the injective objects.

For a complete hereditary cotorsion pair $(\A, \B)$ the term $\dfrac{\Ch(\B)}{ex\B }$ is equivalent to the derived category $\D(\G)$ of the Grothendieck category, essentially because $\Ch(\B)$ contains the dg-injective complexes. Analogously, if $\G$ has enough projectives, $\dfrac{\Ch(\A)}{ex\A}$ is equivalent to the derived category $\D(\G)$ of the Grothendieck category, since $\Ch(\A)$ contains the dg-projective complexes.

By Neeman~\cite{Nee90} the derived category of an idempotent complete exact category $\C$ is defined as the quotient $\dfrac{\Ch(\C)}{\widetilde{\C}}$, where $\widetilde{\C}$ denotes the class of unbounded complexes acyclic in $\C$, meaning that the differentials factor through short exact sequences in $\C$.
Thus it would be important to get recollements analogous to the above ones, but with right term $\dfrac{\Ch(\B)}{\widetilde{\B}}$ or $\dfrac{\Ch(\A)}{\widetilde{\A}}$, that is the derived categories of the exact categories $\B$ or $\A$. Of course, if $\widetilde{\B}=ex\B$ or $\widetilde{\A}=ex{\A}$, the recollements above degenerate into
\vskip0.7cm
\[
\xymatrix{\dfrac{\widetilde{\B}}{\sim} \ar[rr]^{inc} &&{K(\B)} \ar@/^2pc/[ll]\ar@/_2pc/[ll] \ar[rr]^{Q}
&&{\dfrac{\Ch(\B)}{\widetilde{\B} }}\ar@/^2pc/ [ll] \ar@/_2pc/ [ll] }
,\]
\vskip0.7cm

\[
\xymatrix{\dfrac{\widetilde{\A}}{\sim} \ar[rr]^{inc} &&{K(\A)} \ar@/^2pc/[ll]\ar@/_2pc/[ll] \ar[rr]^{Q}
&&{\dfrac{\Ch(\A)}{\widetilde{\A} }}\ar@/^2pc/ [ll] \ar@/_2pc/ [ll] }
.\]
\vskip0.7cm

The only non degenerate example for the case $\widetilde{\B}\subsetneq ex\B$ of which we are aware, is given by the cotorsion pair $(\A, \FpInj)$ over a coherent ring where $\FpInj$ denotes the class of Fp-injective modules, that is the right Ext-orthogonal to the class of finitely presented modules. This follows by \sto's results in \cite{sto14} which we are able to slightly generalize in Proposition~\ref{P:B-inj-n}.

Symmetrically, it seems there are very few non degenerate examples of such recollements for the case $\widetilde{\A}\subsetneq ex\A$. The more important one follows by the celebrated Neeman's result in \cite{Nee08} and it is the case when $\A$ is class of flat modules. We show a slight generalization of this situation in Proposition~\ref{P:A-flat}.%

From the results in Section~\ref{S:2} and the results in a recent paper~\cite{BCIE} we obtain many examples of cotorsion pairs $(\A, \B)$ in module categories satisfying the condition $\widetilde{\B}=ex\B$. These include tilting and cotilting cotorsion pairs, the closure of the cotorsion pair generated by the compact objects of finite projective dimension and the cotorsion pair $(\F, \C)$ of the flat and cotorsion modules.

\section{Preliminaries}

\subsection{Cotorsion pairs}\label{S:cot-pair}

The notion of an \emph{exact category} was introduced by Quillen in \cite{Q}.
 An exact category is an additive category  $\C$ endowed with  a collection $\Phi$ of kernel-cokernel pairs  satisfying some axioms which allow to work with the sequences  in $\Phi$ as if they were exact sequences in an abelian category. An element $E\in \Phi$ is denoted by $0\to A\overset{i}\to B\overset{d}\to C\to 0$ and is called a \emph{conflation or short exact sequence}. The map $i$ is called \emph{inflation} and $p$ is called \emph{deflation}.
In an exact category pushouts (pullbacks) of inflations (deflations) exist and inflations (deflations) are stable under pushouts (pullbacks).

The axioms on conflations allow to define the Yoneda functor $\Ext^i_{\C}(M,N)$ for every pair of objects $M,N$
 in $\C$.
For more details see \cite{Kel90} or \cite{Bu}.

We will deal with weakly idempotent complete (WIC) additive categories, that is categories such that every section has a cokernel or, equivalently, every retraction has a kernel.

Given a class $\X$ of objects in an exact category $\C$, the right orthogonal class $\X^{\perp_1}$ consists of the objects $Y$ such that $\Ext^1_{\C}(X,Y)=0$ for each object $X\in \X$. Similarly, the left orthogonal class $^{\perp_1}\X $ consists of the objects objects $Y$ such that $\Ext^1_{\C}(Y,X)=0$ for each object $X\in \X$.
$\X^\perp$ will denote the class of objects $Y$ such that $\Ext^i_{\C}(X,Y)=0$ for each object $X\in \X$ and each $i\geq 1$. Similarly, for the orthogonal $^\perp\X $.
\begin{defn}\label{D:cotorsion-pair} A pair of classes $(\A, \B)$ in an exact category $\C$ is called a cotorsion pair if
\begin{enumerate}
\item $\A^{\perp_1}=\B$ and $^{\perp_1} \B=\A$.
\item A cotorsion pair is \emph{generated (cogenerated)} by a class $\X$ of objects if $\B=\X^{\perp_1} (\A={^{\perp_1}\X})$.
\item A cotorsion pair $(\A, \B)$ has \emph{enough projectives} if every object $C\in \C$ has a special $\A$-precover, that is there is a short exact sequence $0\to B\to A\to C\to 0$ in $\C$ with $A\in \A$ and $\B\in B$. Dually, we say that  $(\A, \B)$ has \emph{enough injective} if every object $C\in \C$ has a special $\B$-preenvelope, that is there is a short exact sequence $0\to C\to B\to A\to 0$ in $\C$ with $A\in \A$ and $\B\in B$.
\item A cotorsion pair is \emph{complete} when it has \emph{enough injectives} and \emph{enough projectives}.
\item A cotorsion pair is called \emph{hereditary} if $\A^\perp=\B$ and $^{\perp} \B=\A$.
\end{enumerate}
\end{defn}
A class $\C$ of objects in an exact category is \emph{deconstructible} and denoted by $\F ilt\ \clS$, if there is a set $\clS$ of objects such that every object of $\C$ is a transfinite extension of objects of $\clS$ (for more details see \cite[Definition 3.7 and 3.10]{Sto13}).

It is possible to prove, using the so called Small Object Argument, that any cotorsion pair $(\A,\B)$ generated by a set in a category of modules is complete (see \cite{Q} or \cite{ET01}). The argument can be actually extended to Grothendieck categories, provided that $\A$ is generating. We give a precise statement in the following lemma.

\begin{lem}\label{lem1_2} Let $(\A,\B)$ be a cotorsion pair in a Grothendieck category such that $\A$ is generating. Then:
  \begin{enumerate}
    \item $\A$ is generated by a set if and only if it is deconstructible.
    \item If the equivalent conditions in (1) hold, then $(\A,\B)$ is complete.
  \end{enumerate}
\end{lem}

\begin{proof}
  (1) If $\A$ is deconstructible, call $\clS$ the set such that $\A = \F ilt\,\clS$. Then, $\F ilt\,\clS \subseteq {^\perp(\clS^\perp)}$ by Eklof's lemma, but ${^\perp(\clS^\perp)}\subseteq\A$ so they are actually equal, i.e. $(\A,\B)$ is generated by $\clS$.
  Conversely, if $(\A,\B)$ is generated by a set $\clS$, it is also generated by $\clS'=\clS\cup\{G\}$, where $G\in\A$ is a generator. Then, by \cite[Theorem~5.16]{Sto13} $\A$ consists of retracts of $\F ilt \clS$, and by \cite[Proposition~2.9(1)]{St10-deconstr} it is decontructible.

  (2) \cite[Theorem~5.16]{Sto13} actually gives a proof of this statement.
\end{proof}

We will mostly deal with  hereditary cotorsion pairs and
in order to characterize them we recall the following definition.
\begin{defn}\label{D:thick} Let $\C'$ be  a full subcategory of a WIC exact category $\C$.
\begin{enumerate}
\item $\C'$ is \emph{thick} if it is closed under direct summands and has the 2 out of 3 property.
\item $\C'$ is \emph{resolving} in $\C$ if $A\in \C'$ for every exact sequence $0\to A\to B\to C\to 0$ in $\C$ with $B, C\in \C'$.
\item $\C'$ is \emph{coresolving} in $\C$ if $C\in \C'$ for every exact sequence $0\to A\to B\to C\to 0$ in $\C$ with $A, B\in \C'$.
\end{enumerate}
\end{defn}

A complete cotorsion pair $(\A, \B)$ is hereditary \iff $\A$ is resolving or, equivalently, \iff $\B$ is coresolving  (see for instance \cite[Lemma 6.17]{Sto13}).

\subsection{Model structures}\label{S:model}
Model structures on bicomplete abelian categories were introduced by Quillen in \cite{QHtp}.
For the definition of a model structure we refer to the book by Hovey \cite{Hov99} or to the survey \cite{Sto13}.

We only recall that a \emph{model structure} on a category $\C$ consists
of three classes of morphisms Cof, W, Fib called \emph{cofibrations}, \emph{weak equivalences} and \emph{fibrations}, respectively, satisfying certain axioms. An object $X\in \C$ is \emph{cofibrant} if $0\to X$ is a cofibration,  \emph{fibrant} if $X\to 0$ is a fibration and it is \emph{trivial} if $0\to X$ is a weak equivalence.
In particular, the class $\W$ of trivial objects has the 2-out-of-3 property.
A \emph{model category} $\C$ is an abelian cocomplete category with a model structure.
 The \emph{homotopy category} $\Ho \C$ is obtained by  formally inverting all morphisms in W.

A tremendous impulse to the theory was given by  Hovey who discovered ~\cite{Hov07} a bijective correspondence between abelian model structures and cotorsion pairs in abelian categories.
In \cite{G5} Gillespie extended the notion of model structures on exact categories and proved the analogous of Hovey's correspondence in this more general setting.
 We recall the basic notions and results.
  \begin{defn}\label{D:exact-model} An \emph{exact model structure} on an exact category $\C$ is a model structure such that cofibrations (fibrations) are the inflations (deflations) with cofibrant (fibrant) cokernels (kernels).
\end{defn}
\begin{thm}\label{T:correspondence} (\cite{Hov07}, \cite{G5}) Let $\C$ be a WIC exact category with an exact model structure. Let $\clQ$ be the class of cofibrant objects, $\R$ the class of fibrant objects and $\W$ the class of trivial objects. Then $\W$ is a thick subcategory of $\C$, and $(\clQ, \R\cap \W)$ and $(\clQ\cap\W, \R)$ are complete cotorsion pairs in $\C$.
Moreover, given three classes $\W, \clQ, \R$ such that $\W$ is thick in $\C$, $(\clQ, \R\cap \W)$ and $(\clQ\cap\W, \R)$ are complete cotorsion pairs in $\C$, then there is an exact model structure on $\C$ where $\clQ$ are the cofibrant objects, $\R$ are the fibrant objects and $\W$ the trivial objects.
\end{thm}

We denote by $\Ch(\C)$ the category of cochain complexes $X$ with component $X^n\in \C$ in degree $n$ and with differential $d^n_X\colon X^n\to ^{n+1}$ for every $n\in \bbZ$. The morphisms in $\Ch(\C)$ are the usual cochain maps. The suspension is denoted by $[-]$.
 If $\C$ is an exact category, then $\Ch(C)$ is equipped with the exact structure where the short exact sequences are the sequences which are exact in each degree. We can also consider the exact structure on $\Ch(\C)$ where the short exact sequences are degreewise splitting. $\Ext_{dw}(X, Y)$ denotes the Yoneda group of these degreewise splitting sequences.

 For every object $C\in \C$, $S^n(C)$ denotes the complex with entries $0$ for every $i\neq n$ and with $C$ in degree $n$; $D^n(C)$ denotes the complex with $C$ in degrees $n$ and $n+1$ and $0$ elsewhere and with differential $d^n$ being the identity on $C$. The homotopy category $\K(\C)$ has the the same objects as $\Ch(\C)$ and the equivalence classes of cochain maps under the homotopy relation as morphisms.

Given two complexes $X$ and $Y$, the complex $\HOM(X, Y)$ is defined as the complex of abelian groups having $\prod\limits_{p\in \bbZ}\Hom_{\C}(X^p, Y^{n+p})$ in degree $n$  and with differential $d_H(f)=d_Y\circ f -(-1)^nf\circ d_X$. The $n^{\text th}$-cohomology of  $\HOM(X, Y)$ is given by $\Hom_{\K(C)} (X, Y[n])$.

We recall the useful and important formula
\[(\ast)\quad \Ext^1_{dw}(X, Y) \cong \Hom_{\K(\C)}(X,Y[1]).\] %
\begin{nota}\label{N:notation} (Following Gillespie's notations) Let $\A$ be a class of objects in an abelian category $\C$.
 Define the following classes of cochain
  complexes in $\Ch(\C)$:

  \begin{itemize}
    \item
      $dw\A$ is the class of all complexes $X\in\Ch(\C)$
      such that $X^n\in\A$ for all $n\in\bbZ$.  $\Ch(\A)$ will denote the full subcategory of  $\Ch(\C)$ with objects in $dw \A$.
       \item
      $ex\A$ is the class of all acyclic complexes in $dw\A$.
       \item
      $\tilde{\A}$ is the class class of all complexes $X$ in $
      ex\A$ with the cycles $Z^n(X)$ in $\A$ for all $n\in \bbZ$.
   $\Ch_{ac}(\A)$ will denote the full subcategory of  $\Ch(\C)$ with objects in $\tilde\A$.

    \item  If $(\A, \B)$ is a cotorsion pair in $\C$, then:

      $dg\A$ is the class of all complexes $X\in dw\A$
      such that  any morphism
      $f:X\to Y$ with $Y\in\tilde {\B}$ is null homotopic. Since $\Ext^1_{\C}(A^n, B^n)=0$ for every $n\in \bbZ$ formula $(\ast)$ shows that $dg\A={}^\perp{} \tilde {\B}$.

      Similarly, $dg\B$ is the class of all complexes $Y \in dw\B$
      such that  any morphism
      $f:X\to Y$ with $X\in\tilde{\A}$  is null homotopic. Hence $dg\B=\tilde{\A}{}^\perp{}$.
  \end{itemize}
\end{nota}
\begin{lem}\label{L:HOM} Let $\C$ be an abelian category and let $0\to X\to Y\to Z\to 0$ be a short exact sequence of complexes in $\Ch(\C)$ with the degreewise exact structure. For every $A\in \Ch(\C)$ the sequence:
\[0\to \HOM(A, X)\to  \HOM(A, Y)\to \HOM(A, Z)\]
is an exact sequence of complexes in $\Ch(\bbZ)$ and it is also right exact provided that $\Ext_{\C}(A^n, X^n)=0$ for all $n\in \bbZ$.
 Dually, for every $B\in \Ch(\C)$ the sequence:
\[0\to \HOM(Z, B)\to  \HOM(Y, B)\to \HOM(X, B)\]
is an exact sequence of complexes in $\Ch(\bbZ)$ and it is also right exact provided that $\Ext_{\C}(Z^n, B^n)=0$ for all $n\in \bbZ$.
\end{lem}
\begin{proof} Immediate from the definition of the complex $\HOM$.
\end{proof}
\section{Hereditary cotorsion pairs in Grothendieck categories}

We recall some results which will be used throughout. Their proof can be found in \cite{St10-deconstr}, \cite{Sto13},  \cite{G4}, \cite{G6}.

 \begin{prop}\label{P:description-tilde} (\cite[Proposition 7.13, 7.14]{Sto13} Let $(\A, \B)$ be a complete cotorsion pair in an abelian category $\C$. The following hold true
  \begin{enumerate}
  \item A complex $Y$ belongs to $\tilde\B$ \iff $\Ext^1_{\Ch(\C)}(S^n(A), Y)=0$ for every $n\in \bbZ$ and every $A\in \A$.
  \item A complex $X$ belongs to $\tilde\A$ \iff $\Ext^1_{\Ch(\C)}(X, S^n(B))=0$ for every $n\in \bbZ$ and every $B\in \B$.
  \item If $\C$ is a Grothendieck category and $(\A, \B)$ is a complete hereditary cotorsion pair, then $(\tilde\A, dg\B)$ and $(dg \A, \tilde \B)$ are complete hereditary cotorsion pairs in $\Ch(\C)$.

    \item $(dg\A, \E, dg\B)$ is a model structure on $\Ch(\C)$ with the acyclic complexes $\E$ as trivial objects. In particular, $dg\A\cap\E=\tilde \A$ and $dg\B\cap\E=\tilde \B$.
  \end{enumerate}
  \end{prop}
  \begin{proof}
 (1) and (2) are proved in \cite[Lemma 7.13]{Sto13}. (3) is proved in \cite[Proposition 7.14]{Sto13}.
  (4) follows by (3) and by Hovey's correspondence (see Theorem~\ref{T:correspondence}).
  \end{proof}
   \begin{prop}\label{P:complete?} Let $(\A, \B)$ be a complete cotorsion pair in an Grothendieck category $\G$.
 \begin{enumerate}
  \item $(dw\A, dw\A^\perp)$ and $(^\perp dw\B, dw\B)$ are cotorsion pairs in $\Ch(\G)$.
  \item If $(\A, \B)$ is generated by a set, then so is $(^\perp dw\B, dw\B)$.
  \item If $(\A, \B)$ is generated by a set, then so is $(dw\A, dw\A^\perp)$.

  \item If $\A$ contains a generator of $\G$ with finite projective dimension,
  then $(^\perp ex\B, ex\B)$ is a cotorsion pair in $\Ch(\G)$. If moreover, $(\A, \B)$ is generated by a set, then so is $(^\perp ex\B, ex\B)$.
  \item $(ex\A, ex\A^\perp)$ is a cotorsion pair. Moreover, if $(\A, \B)$ is generated by a set, then so is $(ex\A, ex\A^\perp)$.

   \end{enumerate}
  \end{prop}
 \begin{proof}
  (1) is proved in \cite[Proposition 3.2]{G4},

  (2) is proved in \cite[Proposition 4.4]{G4}.

  (3) is proved as follows: by Lemma~\ref{lem1_2}, $\A$ is deconstructible and by \cite[Theorem 4.2]{St10-deconstr} so is $dw\A$. Moreover, $dw\A$ contains a generator, so $(dw\A,dw\A^\perp)$ is generated by a set by Lemma~\ref{lem1_2}.

  The first part of (4) is proved in \cite[Proposition 3.3]{G4}; the second part in \cite[Proposition 4.6]{G4}.

  The first part of (5) is again proved in \cite[Proposition 3.3]{G4}; for the second part we argue as in the proof of \cite[Proposition 7.3]{G6}. $ex\A=dw\A\cap \E$, where $\E$ is the class of acyclic complexes. By \cite[Theorem 4.2]{St10-deconstr} $\E$ and $dw \A$ are deconstructible, hence $ex\A$ is deconstructible by \cite[Proposition 2.9]{St10-deconstr}.
  Moreover, $ex\A$ contains a generator, so $(ex\A,ex\A^\perp)$ is generated by a set by Lemma~\ref{lem1_2}.
  \end{proof}

\begin{rem} If $(\A, \B)$ is a complete hereditary cotorsion pair in an abelian category, then the complete cotorsion pairs defined in the above proposition are hereditary, too.
\end{rem}

\begin{lem}\label{L:ex-tilde} Let $(\A, \B)$ be a complete hereditary cotorsion pairs in a Grothendieck category $\G$. Then $ex\B=\tilde \B$ if and only if $dw\B=dg\B$.
Dually $ex\A= \tilde \A$ if and only if $dw\A=dg\A$
 \end{lem}
  \begin{proof} Assume that $ex\B= \tilde \B$ and let $Y\in dw\B$. We have to show that $\Ext^1_{\Ch(\C)}(X, Y)=0$ for every $X\in \tilde \A$.  Equivalently we have to show that the complex $\HOM(X, Y)$ is exact  for every $X\in \tilde \A$. Since, $(\tilde\A, dg\B)$ is a complete cotorsion pair in $\Ch(\C)$  there is a short exact sequence
\[ 0\to Y\to Z\to V\to 0\]
with $Z\in dg\B$ and $V\in \tilde\A$. Now, $\B$ is coresolving,  hence $V\in \tilde\A\cap dw\B=\tilde\A\cap ex\B$ and the last is $\tilde\A\cap \tilde\B$ by  assumption. Thus, $V$ is contractible, hence null homotopic.
By Lemma~\ref{L:HOM} we have a short exact sequence
\[0\to \HOM(X, Y)\to  \HOM(X, Z)\to \HOM(X, V)\to 0\]
for every $X\in \tilde \A$. The second and the third nonzero terms are exact, hence also $ \HOM(X, Y)$ is exact.

Conversely, assume that $dw\B=dg\B$ and let $Y\in ex\B$. Then $Y\in dw\B\cap\E=dg\B\cap \E$ and by Proposition~\ref{P:description-tilde}~(4), $Y\in \tilde\B$.

The dual statement is proved in similar ways.
\end{proof}

\section{Cotorsion pairs $(\A, \B)$ satisfying $ex\B=\tilde{\B}$}\label{S:2}
We are interested in describing cotorsion pairs $(\A, \B)$ such that $ex\B=\tilde \B$ or $ex\A=\tilde \A$, since in these cases we have the following important consequences on the corresponding model structures.
\begin{cor}\label{C:cofibrant-fibrant}  Let $(\A, \B)$ be a complete hereditary cotorsion pairs in a Grothendieck category $\G$. The following hold true:
\begin{enumerate}
\item  If $ex\B=\tilde \B$, then $(dg\A, \E, dw\B)$ is a model structure in $\Ch(\G)$ for which the fibrant objects are exactly the complexes with components in $\B$.
\item If $ex\A=\tilde \A$, then $(dw\A, \E, dg\B)$ is a model structure in $\Ch(\G)$ for which the cofibrant objects are exactly the complexes with components in $\A$.
\end{enumerate}
\end{cor}
\begin{proof} Follows by Proposition~\ref{P:description-tilde}~(4) and by Lemma~\ref{L:ex-tilde}.
\end{proof}

We say that an object $M$ in a Grothendieck category $\G$ has projective dimension at most $n$ if $\Ext_{\G}^i(M, -)$ vanishes for every $i>n$ and we denote by $\clP_n$ the class of objects of projective dimension at most $n$.
 Analogously, $M$ has injective dimension at most $n$ if $\Ext_{\G}^i(-, M)$ vanishes for every $i>n$ and we denote by $\I_n$ the class of objects of  injective dimension at most $n$. We denote by $\clP=\bigcup_n\clP_n$ the class of objects with finite projective dimension and by and $\I=\bigcup_n\I_n$ the class of objects with finite injective dimension.

  \begin{prop}\label{P:finite-proj-dim} Let $(\A, \B)$ be a complete hereditary cotorsion pair in a Grothendieck category $\G$ and let $Y$ be an acyclic complex with terms in $\B$. The following hold true:
  \begin{enumerate}
  \item If $M$ is an object in $\A$ with finite projective dimension, then the cycles $Z^j(Y)$ of $Y$ belong to $M{}^\perp{}$.
  \item If $\A\subseteq \clP$, then $Y\in \tilde \B$, hence $ex\B=\tilde \B$.

  In particular, in the abelian model structure corresponding to the cotorsion pair $(\A, \B)$ by Theorem~\ref{T:correspondence}, $dw \B$ is the class of fibrant objects.
  \end{enumerate}

   Dually, let $X$ be an acyclic complex with terms in $\A$. Then:
   \begin{enumerate}
  \item[(3)] If $N$ is an object in $\B$ with finite injective dimension, then the cycles $Z^j(X)$ of $X$ belong to ${^\perp N}$.
  \item[(4)] If $\B\subseteq \I$, then $X\in \tilde \A$, hence $ex\A=\tilde \A$.

    In particular, in the abelian model structure corresponding to the cotorsion pair $(\A, \B)$ by Theorem~\ref{T:correspondence}, $dw \A$ is the class of cofibrant objects.

  \end{enumerate}

  \end{prop}

  \begin{proof} (1) Clearly it is enough to verify that the $0$-cycle $Z^0$ of $Y$ is in $M{}^\perp{}$.  Consider the exact complex
  \[\dots Y^{-n}\to\dots  \to Y^{-2}\to Y^{-1}\to Z^0\to 0,\]
 If $M$ is in $ \A$, then $\Ext^j_{\G}(M, Y^n)=0$ for every $n\in \bbZ$ and every $j\geq 1$. A dimension shifting argument gives $\Ext^i_{\G}(M, Z^0)\cong \Ext^{i+k}_{\G}(M, Z^{-k})$, for every $k\geq 1$. Hence by the finiteness of the projective dimension of $M$   we conclude that $\Ext^i_{\G}(M, Z^0)=0$ for every $i\geq 1$.

 (2) The first statement  follows by~(1). The second statement follows by Corollary~\ref{C:cofibrant-fibrant}.

 The proof of the dual statement is obtained by considering the acyclic complex:
 \[0\to Z^0\to X^0\to X^1\to \dots\to X^n\to \dots\]
 and using dimension shifting for the functor $\Hom_{\G}(- ,N)$.

\end{proof}
We consider now the particular case of a module category and
 we exhibit some situations in which the assumptions of the previous proposition are satisfied.

 Recall that $T$ is an $n$-tilting $R$-module if it has projective dimension at most $n$, $\Ext^i_R(T, T^{(\lambda)})=0$ for every cardinal $\lambda$ and every $i\geq 0$, and  the ring $R$ has a finite coresolution with terms in $\Add T$, where $\Add T$ denotes the class of direct summands of direct sums of copies of $T$.
 The cotorsion pair generated by $T$ is called $n$-tilting cotorsion pair.

 Dually, an $R$-module $C$ is $n$-cotilting if it has injective dimension at most $n$, $\Ext^i_R(C^{\lambda}, C)=0$ for every cardinal $\lambda$ and every $i\geq 0$, and  an injective cogenerator has a finite resolution with terms in $\Prod C$, where $\Prod C$ denotes the class of direct summands of direct products of copies of $C$.
 The cotorsion pair cogenerated by $C$ is called $n$-cotilting cotorsion pair.

  \begin{prop}\label{P:tilt-cotil} If $(\A, \B)$ is an $n$-tilting cotorsion pair in $\Modr R$, then $ex\B= \tilde \B$ and $dw\B=dg\B$. Hence there is a model structure in $\Ch(R)$ in which the fibrant objects are the complexes with components in the  $n$-tilting class $\B$ and the trivial objects are the acyclic complexes.

 Dually, if $(\A, \B)$ is an $n$-cotilting cotorsion pair in $\Modr R$, then $ex\A= \tilde \A$ and $dw\A=dg\A$.
 Hence there is a model structure in $\Ch(R)$ in which the cofibrant objects are the complexes with components in the  $n$-cotilting class $\A$ and the trivial objects are the acyclic complexes.
 \end{prop}

\begin{proof} If $(\A, \B)$ is a tilting (cotilting) cotorsion pair, then $\A\subseteq \clP_n$ ($\B\subseteq \I_n$), by \cite[Lemmas 13.10, 15.4]{GT12}. Hence the conclusion follows by Proposition~\ref{P:finite-proj-dim}.
\end{proof}
To exhibit other examples of cotorsion pairs $(\A, \B)$ satisfying the condition $ex\B=\tilde \B$ we use the notion of the closure of a cotorsion pair.

 Recall that a cotorsion pair  $(\A, \B)$ is \emph{closed} if $\A$ is closed under direct limits. Consider the lattice of cotorsion pairs, with respect to inclusion on the left component.
Since the cotorsion pair $(\Modr R, \Inj)$ is closed and the meet of closed cotorsion pairs is closed (see e.g. \cite{AT} or \cite[Lemma 6.1]{G6}), every cotorsion pair $(\A, \B)$ is contained in a smallest closed cotorsion pair, called the closure of $(\A, \B)$
\begin{nota}\label{N:notations} Let $R$ be a ring.
\begin{enumerate}
\item We denote by $\modr R$ the class of modules $M$ admitting a projective resolution of the form
\[\dots\to P_i\to P_{i-1}\to \dots\to P_1\to P_0\to M\to 0,\]
with $P_j$ finitely generated  for every $j \geq 0$.
\item For every $n\geq 0$, denote by $\clP_n(\modr R)$  the class $\clP_n\cap\modr R$ and by $\clP(\modr R)$  the class $ \clP\cap\modr R$.
\item The \emph{little finitistic dimension} of $R$ is the supremum of the projective dimension of modules in $\modr R$ having finite projective dimension.
\item The \emph{big projective (flat) finitistic dimension}  of $R$ is  the supremum of the projective (flat) dimension of modules having finite projective (flat) dimension.
\item Denote by $(\A^{\omega}, \B_{\omega})$ the complete hereditary cotorsion pair generated by $\clP(\modr R)$. By \cite[Theorem 2.3, Corollary 2.4]{AT}, its closure
\[(\A^{\infty}, \B_{\infty})
\]  is a complete cotorsion pair cogenerated by the class of pure injective modules belonging to $\B_{\omega}$, hence it is hereditary, since cosyzygies of pure injective modules of $\B_{\omega}$ are in $\B_{\omega}$. Moreover, $\A^{\infty}=\varinjlim \A^{\omega}=\varinjlim \clP(\modr R)$ and it is closed under pure epimorphic images.
\end{enumerate}
\end{nota}

\begin{rem}\label{R:sup-flat-if-cot} \begin{enumerate}\item Since $\varinjlim \clP_0(\modr R)$ is the class of flat modules, $\A^{\infty}$ contains all flat modules and it coincides with the class of flat modules if and only if every module in $\clP_n(\modr R)$ is projective, i.e. if the little finitistic dimension of $R$ is $0$.
\item By part (1), $\B_{\infty}$ is contained in the class of cotorsion modules and it is properly contained in it whenever the little finitistic dimension of $R$ is greater than $0$.
\item Moreover,  $\clP_1\subseteq \varinjlim \clP_1(\modr R)$, hence $\clP_1\subseteq \A^{\infty}$.
\item By \cite[Theorem 6.7~(vi)]{BH09}, if $R$ has a classical ring of quotients $Q$ such that $Q$ is Von Neumann regular or has big finitistic flat dimension $0$, then $\varinjlim \clP_1$ coincides with the class $\F_1$ of  modules of flat dimension at most $1$.
Hence $\A^{\infty}$ contains $\F_1$ and $ \B_{\infty}$ is contained in the class $\F_1{}^\perp{}$ which is also called the class of weakly injective modules (see \cite{FuLee1} and \cite{FuLee2}).
 In particular, this applies to any commutative ring such that the total quotient ring is a perfect ring or a Von Neumann regular ring.
\end{enumerate}
\end{rem}
  \begin{prop}\label{P:semihereditary} Let $R$ be a (coherent) ring. The class $\B_{\infty}$ coincides with the class of injective right  $R$-modules if and only if every module in $\modr R$ (every finitely presented module) has finite projective dimension. In particular, this applies to rings with finite little finitistic dimension and thus to right semihereditary rings.
  \end{prop}

 \begin{proof}  $ \B_{\infty}$ coincides with the class of injectives if and only if $\A^{\infty}=\Modr R$. If every module in $\modr R$  has finite projective dimension, then $\A^{\infty}=\Modr R$, since $\A^{\infty}$ is closed under direct limits. Conversely, if $\A^{\infty}=\Modr R$, then every finitely presented right module $X$ belongs to $\varinjlim \clP(\modr R)$, hence it is a summand of a finite direct sum of modules in $ \clP(\modr R)$. Thus $X$ has finite projective dimension and so does every module in $ \modr R$.

 The last statement follows easily. In particular,  if $R$ is right semihereditary, then every finitely presented right $R$-module has projective dimension at most one.\end{proof}

  We show now that $ex\B_{\infty}=\widetilde{\B_{\infty}}$. To this aim we apply the results proved in a recent paper ~\cite{BCIE} about periodic modules.
   Recall that a module $M$ is periodic with respect to a class $\C$ if there exists a short exact sequence $0\to M\to C\to M\to 0$ with $C\in \C$. A module $M$ is Fp-injective if $\Ext^1_R(X, M)=0$ for every finitely presented module $X$.

\begin{fact}\label{F:periodic}
 \begin{enumerate}
\item \cite[Proposition 3.8~(1)]{BCIE} every Fp-injective $\Inj$-periodic module is injective.
 \item \cite{EFI} If $\C$ is  a class closed under direct sums or direct products and  $\D$ is a class closed under direct summands, then the following are equivalent:
\begin{enumerate}
\item[(a)] Every cycle of an acyclic complex with components in $\C$ belongs to $\D$.
\item[(b)]  Every $\C$-periodic module belongs to $\D$.
\end{enumerate}

\end{enumerate}
\end{fact}

 \begin{prop}\label{P:ex-B-infty} The cotorsion pair $(\A^{\infty}, \B_{\infty})$ from Notation~
 \ref{N:notations}~(4) satisfies $ex\B_{\infty}=\widetilde{\B_{\infty}}$.
\end{prop}
\begin{proof}  Let $M$ be a $\B_{\infty}$-periodic module. By \cite[Lemma 3.4]{BCIE} ${}^\perp M\supseteq \clP(\modR)$. As mentioned in Notation~
 \ref{N:notations}~(4), the class  $\A^{\infty}$ coincides with $\varinjlim \clP(\modR)$ and is closed under pure epimorphic images. By \cite[Theorem 3.7]{BCIE} ${}^\perp M\supseteq \A^{\infty}$, hence $M\in\B_{\infty}$. By Fact~\ref{F:periodic}~(2), $ex\B_{\infty}=\widetilde{\B_{\infty}}$ in $\Ch(R)$.
  \end{proof}

As a corollary we get  an improvement of ~\cite[Corollary 5.9]{Stopurity} in the case of a module category, since $ \B_{\infty}$ is in general properly contained in the class of cotorsion modules.
\begin{cor}\label{C:exInj}
Let $Y$ be an acyclic complex with injective components. Then every cycle of $Y$ belongs to $ \B_{\infty}$, hence $Y\in \widetilde{ \B_{\infty}}$.
\end{cor}
\begin{proof} By assumption $Y\in ex\B_{\infty}$, hence the conclusion follows by Proposition~\ref{P:ex-B-infty}.
 \end{proof}
 The next properties will be used in Section~\ref{S:tildeB}.

  \begin{lem}\label{L:Cot-inj} Let $(\A,\B)$ be a complete hereditary cotorsion pair in a Grothendieck category $\G$.
Let $\Inj$ denote the class of injective objects of $\G$. The following hold true:
\begin{enumerate}
\item[(1)] $\B{}^\perp{}\cap \B\subseteq \Inj$ and $\widetilde{\B}{}^\perp{}\cap dw \B\subseteq dw \Inj$.
\item[(2)]  $\widetilde{\B}{}^\perp{}\cap dw \B=dw \Inj$ \iff $\tilde\B\subseteq {}^\perp{} dw\Inj$.
\end{enumerate}
Moreover, if $\G=\Modr R$ and $\B$ contains the class $ \B_{\infty}$ defined in Notation~\ref{N:notations}~(5) then
\begin{enumerate}
\item[(3)] $\widetilde{\B}{}^\perp{}\cap dw \B=dg \Inj$.
\item[(4)] ${}^\perp{} dg\Inj\cap dw\B=ex \B$
\end{enumerate}

\end{lem}
\begin{proof}
(1) Let $B\in \B{}^\perp{}\cap \B$ and consider an exact sequence $0\to B\to I\to I/B\to 0$ with $I\in \Inj$. Then $I/B\in \B$, since $\B$ is coresolving, hence the sequence splits and $B$ is injective.

If $B\in \B$, then $D^n(B)\in \widetilde{\B}$ for every $n\in \bbZ$ and by \cite[Lemma 3.1]{G3}, $\Ext^1_{\Ch(\G)}(D^n(B), Y)\cong \Ext^1_{\G}(B, Y^n)$, for every complex $Y$. Thus if $Y\in \widetilde{\B}{}^\perp{}\cap dw \B$, then $Y^n\in \B{}^\perp{}\cap\B$ for every $n\in \bbZ$. By the above we conclude that $Y\in dw \Inj$.

(2) If $\tilde\B\subseteq {}^\perp{} dw\Inj$, then $\tilde\B{}^\perp{}\supseteq  ({}^\perp{} dw\Inj){}^\perp{}=dw\Inj$, by \cite[Proposition 4.4]{G4}, hence by part (1) $\tilde{\B}{}^\perp{}\cap dw \B=dw \Inj$.

Conversely, if $\widetilde{\B}{}^\perp{}\cap dw \B=dw \Inj$, then $dw\Inj\subseteq \tilde\B^\perp$, hence $\tilde\B\subseteq {}^\perp{}(\tilde\B{}^\perp{})\subseteq ^\perp{}dw\Inj.$

(3) We show the inclusion $\widetilde{\B}{}^\perp{}\cap dw \B\subseteq dg \Inj$. Let $Y\in \widetilde{ \B}{}^\perp{}\cap dw \B$; using the complete cotorsion pair $(\E, dg \Inj)$ in $\Ch(R)$ we can consider a  short exact sequence $(\ast)\quad 0\to Y\to dg I\to E\to 0$ with $dg I\in dg \Inj$ and $E$ an exact complex. By part (1) the sequence is degreewise splitting hence $E^n$ is an injective module for every $n\in \bbZ$ which means that $E\in ex \Inj$.
 By  Corollary~\ref{C:exInj}, $ex \Inj\subseteq \widetilde{ \B_{\infty}}\subseteq\widetilde{ \B}$, hence the sequence $(\ast)$ splits showing that $Y\in dg \Inj$.

 The other inclusion is  obvious since  ${}^\perp{} dg \Inj$ in $\Ch(R)$ is the class of acyclic complexes $\E$ and $\E\supseteq \widetilde{ \B}$.

 (4) Obvious, since ${}^\perp{} dg \Inj=\E$.  \end{proof}
\begin{rem} If  $\G$ has enough projective objects, then the dual of the statements in Lemma~\ref{L:Cot-inj}~(1) and (2) hold  substituting the right orthogonal with the left orthogonal and $\Inj$ with $\Proj$.
\end{rem}

 \section{Cotorsion pairs in exact categories}

 We state a result valid in general for cotorsion pairs in exact categories.
 \begin{prop}\label{P:inducing} Let $(\A, \B)$ be a (hereditary) complete cotorsion pair in an exact category $\C$ and let $\D$ be an extension closed subcategory of $\C$ with the exact structure induced by that of $\C$. If $\D$ contains $\A$ and is resolving in $\C$ or if $\D$ contains $\B$ and is coresolving in $\C$, (see Definition~\ref{D:thick}), then $(\A\cap\D, \ \B\cap\D)$ is a (hereditary) complete cotorsion pair in the exact category $\D$.

 \end{prop}
 \begin{proof} We prove the statement in  case $\D\supseteq \B$, the other case being similar. First we show that $(\A\cap\D, \B)$ is a cotorsion pair in $\D$.
 Clearly ${}^\perp{}\B= \A\cap \D$ in $\D$ and also $(\A\cap \D){}^\perp{} \supseteq \B$. We show that $(\A\cap \D){}^\perp{} =\B$ in $\D$.
 Let $D\in \D$ be such that $\Ext^1(X, D)=0$ for every $X\in \A\cap \D$. Since $(\A, \B)$ is complete, there is an exact sequence $0\to D\to B\to A\to 0$ in $\C$, with $B\in \B$ and $A\in \A$. Since $\D$ is coresolving  in $\C$ and contains $\B$, we have that $A\in \D$, hence $A\in A\cap\D$ showing that the exact sequence splits, thus $D\in \B$.

To show that $(\A\cap\D, \B)$ is complete, let $(\ast)\quad 0\to B\to A\to D\to 0$ be a special $\A $-precover of an object $D\in \D$, then $A\in \A\cap\D$, since $\D$ is extension closed, hence $(\ast)$ is a special $\A\cap \D $-precover of $D$. If $ (\ast\ast):\quad 0\to D\to B\to A\to 0$ is a special $\B $-preenvelope of $D\in \D$, then $A\in \D$ since $\D$ is coresolving, hence $(\ast\ast)$ is  special $\B$-preenvelope of $D$ \wrt to  $(\A\cap\D, \B)$. \end{proof}

The notions of injective and projective Hovey triples and of injective and projective cotorsion pairs have been introduced in \cite{G7} following the analogous concepts defined in \cite{Beck14} and \cite{G6}.

\begin{defn}  Assume that a WIC exact category $\C$ has enough injective objects.
A complete cotorsion pair $(\A, \B)$ in $\C$ is an \emph{injective cotorsion pair} if $\A$ is thick and contains the injective objects.
Symmetrically, assume that $\C$ has enough projective objects.
A complete cotorsion pair $(\A, \B)$ in $\C$ is a \emph{projective cotorsion pair} if $\B$ is thick and contains the projective objects.

\end{defn}
Thus, an injective cotorsion pair $(\A, \B)$ corresponds to the model structure $(\C, \A, \B)$ where all objects are cofibrant and a projective cotorsion pair $(\A, \B)$ corresponds to the model structure $ (\A, \B, \C)$ where all objects are fibrant.

\begin{defn} (\cite[Definition 4.3]{G7}) Let $\C$ be a WIC Frobenius category.
An injective complete cotorsion pair $(\A, \B)$ in $\C$ is called a \emph{localizing cotorsion pair}. If $(\A, \B)$ and $(\B, \D)$ are injective cotorsion pairs in $\C$, then $(\A, \B, \D)$ is called a \emph{localizing cotorsion triple} in $\C$.
\end{defn}

\begin{rem}\label{R:inj-proj-model} By \cite[Proposition 4.2]{G7}, a complete hereditary cotorsion pair  in a WIC Frobenius category is an injective cotorsion pair if and only if it is a projective cotorsion pair if and only if $\A$ is thick if and only if $\B$ is thick.

\end{rem}

From now on $\G$ will be a Grothendieck category.

For every complete cotorsion pair $(\A, \B)$ in a Grothendieck category $\G$ the classes $\A$ and $\B$ are extension closed subcategories of $\G$, hence they inherit the exact structure from the abelian structure of $\G$.

Moreover, it is obvious that they are idempotent complete.

 It is well known that a Grothendieck category has enough injectives.
When needed we will assume that $\G$ has enough projectives and enough flat objects.

We will denote by $\Inj$ and $\Proj$ the classes of injective and projective objects, respectively; by $\Flat$ the class of flat objects and by $\Cot$ the class of cotorsion objects.  We have the complete hereditary cotorsion pairs $(\Proj, \G)$, $(\G, \Inj)$ and $(\Flat, \Cot)$, hence the four classes defined above are exact subcategories of $\G$.

We first collect some well known facts.
\begin{fact}\label{F:proj-inj} Let $(\A, \B)$ be a complete hereditary cotorsion pair in $\G$. The following hold true:
\begin{enumerate}
\item $\A$ has enough injectives and projectives: the projectives are the same as in $\G$ and the injectives are the objects in $\A\cap\B$.
\item $\B$ has enough injectives and projectives: the injectives are the same as in $\G$ and the projectives are the objects in $\A\cap\B$.
\item (\cite[Corollary 2.9]{G7} $\Ch(\A)$ has enough injectives and projectives: the projectives are the same as in $\Ch(\G)$ and the injectives are the contractible complexes with components in $\A\cap\B$.
\item (\cite[Corollary 2.9]{G7} $\Ch(\B)$ has enough injectives and projectives: the injective are the same as in $\Ch(\G)$ and the projectives are the contractible complexes with components in $\A\cap\B$.
\item (\cite[Corollary 2.8]{G7} $\Ch(\A)_{dw}$ and $\Ch(\B)_{dw}$ are Frobenius exact category with the  projective-injective objects being the contractible complexes with terms in $\A$ or $\B$ respectively.
\end{enumerate}
\end{fact}

\begin{rem}\label{R:induced} If $(\A, \B)$ is a complete hereditary cotorsion pair in a Grothendieck category $\G$, Proposition~\ref{P:inducing} tells us that $(\B, \Inj)$ and $(\A\cap \B, \B)$  are complete hereditary cotorsion pairs in the exact category $\B$; $(\Proj, \A)$ and $(\A, \A\cap \B)$ are complete hereditary cotorsion pairs  in the exact category $\A$.
\end{rem}

 The next Proposition~\ref{P:prop7.3} is a generalization of \cite[Proposition 7.3]{G7} which was formulated for the case of the cotorsion pair $(\F, \C)$ in a module category. %

Moreover, in Proposition~\ref{P:prop7.3-dual} we state a generalization of the dual of \cite[Proposition 7.3]{G7}.

\begin{prop}\label{P:prop7.3} Let $(\A, B)$ be a complete hereditary cotorsion pair in a Grothendieck category $\G$ and let $(\hat\A, \hat\B)$ be a complete cotorsion pair in $\Ch(\G)$ with $\hat\A\subseteq dw\A$. Assume that $\hat\A$ is thick in the exact category $\Ch(\A)$ and that it contains the contractible complexes with terms in $\A$.
Then, \[\Big(\hat\A, \hat\B\cap dw \A\Big)\] is an injective cotorsion pair in $\Ch(\A)$.
Moreover,  \[\Big(\hat\A, [\hat\B\cap dw \A]_K\Big)\] is a localizing cotorsion pair $\Ch(\A)$ in the Frobenius category $\Ch(\A)_{dw}$, where a complex $X\in \Ch(\A)$ belongs to $[\hat\B\cap dw \A]_K$ if and only if it is chain homotopy equivalent to a complex in $\hat\B\cap dw \A$.
\end{prop}
\begin{proof} The fact that $\Big(\hat\A, \hat\B\cap dw \A\Big)$ is a complete cotorsion pair follows by Proposition~\ref{P:inducing} and it is an injective cotorsion pair by definition and by the assumptions on $\hat\A$. Moreover, $\hat\B\subseteq dw \B$. In fact, for every $n\in \bbZ$ and every  $A\in \A$  the contractible complex $D^n(A)$ is in $\hat\A$, hence $\Ext^1_{\Ch}(D^n(A),  B)=0$, for every $B\in\hat\B$ and then $B^n$ belongs to $\B$, by \cite[Lemma 3.1]{G1}. Hence, a short exact sequence $0\to Y\to Z\to X\to 0$ in $\Ch(\A)$ with $Y\in\hat\B\cap dw \A$ and $X\in \hat\A$ is degreewise splitting. The second statement follows by \cite[Theorem 6.3, Proposition 6.4]{G7}.
\end{proof}
\begin{prop}\label{P:prop7.3-dual} Let $(\A, B)$ be a complete hereditary cotorsion pair in $\G$ and let $(\hat\A, \hat\B)$ be a complete cotorsion pair in $\Ch(\G)$ with $\hat\B\subseteq dw\B$. Assume that $\hat\B$ is thick in the exact category $\Ch(\B)$ and contains the contractible complexes with terms in $\B$.
Then, \[\Big(\hat\A\cap dw \B, \hat\B\Big)\] is a projective cotorsion pair in $\Ch(\B)$.
Moreover,  \[\Big([\hat\A\cap dw \B]_K, \hat\B\Big)\] is a localizing cotorsion pair  in the Frobenius category $\Ch(\B)_{dw}$,  where a complex $X\in \Ch(\B)$ belongs to $[\hat\A\cap dw \B]_K$ if and only if it is chain homotopy equivalent to a complex in  $\hat\A\cap dw \B$.
\end{prop}
\begin{proof}
The fact that $\Big(\hat\A\cap dw \B, \hat\B\Big)$ is a complete cotorsion pair follows by Proposition~\ref{P:inducing} and it is a projective cotorsion pair by definition and by the assumptions on  $\hat\B$.  Moreover, $\hat\A\subseteq dw\A$. In fact, for every $n\in \bbZ$ and every  $B\in \B$  the contractible complex $D^n(B)$ is in $\hat\B$, hence $\Ext^1_{\Ch}(A, D^n(B))=0$, for every $A\in\hat\A$ and then $A^n$ belongs to $\A$, by \cite[Lemma 3.1]{G1}. Hence, a short exact sequence $0\to Y\to Z\to X\to 0$ in $\Ch(\B)$ with $Y\in \hat\B$ and $X\in \hat\A\cap dw \B$ is degreewise splitting.
Note that Theorem 6.3 and Proposition 6.4 in \cite{G7} have  obvious dual statements for projective cotorsion pairs from which the second statement of our proposition follows.\end{proof}

\section{Projective cotorsion pairs in the exact category $\Ch(\B)$}\label{S:B}

For every complete hereditary cotorsion pair $(\A, \B)$ in a Grothendieck category $\G$ we look for cotorsion pairs on the exact category $\Ch(\B)$ of unbounded complexes with terms in $\B$ in order to describe the homotopy category $K( \B)$ that is the quotient $\Ch( \B)/\sim$ where $\sim$ denotes the chain homotopy equivalence.

We start by choosing projective cotorsion pairs in $\Ch(\G)$ satisfying the assumptions of Proposition~\ref{P:prop7.3-dual}.
When needed we assume some extra conditions on the Grothendieck category $\G$, like in example (3) below.

\begin{expl}\label{E:examples-proj} Let $(\A, \B)$ be a complete hereditary cotorsion pair in a  Grothendieck category $\G$.
 \begin{enumerate}
\item The complete hereditary cotorsion pair $(dg \A, \widetilde{\B})$ in $\Ch(\G)$ satisfies the conditions in  Proposition~\ref{P:prop7.3-dual}, hence we have the projective cotorsion pair:
 \[\M_1=\Big(dg{\A}\cap dw \B, \widetilde{\B}\Big)\] in $\Ch(\B)$ and the localizing cotorsion pair $\Big([dg{\A}\cap dw \B]_K,  \widetilde{\B}\Big)$ in $\Ch(\B)_{dw}.$
\item The complete hereditary cotorsion pair $(\widetilde{\A}, dg \B)$ in $\Ch(\G)$ satisfies the conditions in  Proposition~\ref{P:prop7.3-dual}, hence we have the projective cotorsion pair:
 \[\M_2=\Big(\widetilde{\A}\cap dw \B, dg \B\Big)\] in $\Ch(\B)$ and the localizing cotorsion pair $\Big([\widetilde{\A}\cap dw \B]_K, dg (\B)\Big)$ in $\Ch(\B)_{dw}.$

\item If $\A$ contains a generator of finite projective dimension, then by Proposition~\ref{P:complete?}~(3), $({}^\perp{} ex\B, ex\B)$ is a complete hereditary cotorsion pair in $\Ch(\G)$ and it satisfies the conditions in  Proposition~\ref{P:prop7.3-dual}, hence  we have the projective cotorsion pair:
 \[\M_3=({}^\perp{} ex\B\cap dw \B, ex\B)\] in $\Ch(\B)$ and the localizing cotorsion pair $\Big([{}^\perp{} ex\B\cap dw \B]_K, ex\B\Big)$ in $\Ch(\B)_{dw}.$

 \end{enumerate}
 \end{expl}
 \begin{rem}\label{R:thick-1} The three examples above satisfy Proposition~\ref{P:prop7.3-dual} since $\tilde{\B}$, $dg\B$ and $ex\B$  are thick in $\Ch(\B)$ by Lemma~\ref{L:HOM} and they clearly contain the contractible complexes with terms in $\B$.
 \end{rem}

\begin{thm}\label{T:recollement-proj} Let $(\A, \B)$ be a complete hereditary cotorsion pair in a  Grothendieck category $\G$ such that $\A$ contains a generator of finite projective dimension.

The three projective cotorsion pairs in Example~\ref{E:examples-proj} satisfy the conditions of \cite[Theorem 3.5]{G7}, so that we get the recollement:

\[
\xymatrix{\dfrac{\widetilde{\A}\cap dw \B}{\sim} \ar[rr]^{inc} &&\dfrac{dg{\A}\cap dw \B}{\sim} \ar@/^2pc/
[ll]\ar@/_2pc/[ll] \ar[rr]^{Q}
&&{\Ch(\B)/ex\B}\ar@/^2pc/ [ll] \ar@/_2pc/ [ll] }
\]
\vskip 0.7cm

where $\sim$ denotes the homotopy relation associated to the corresponding model structure and coincides with the chain homotopy relation; moreover, $inc$ is the inclusion and $Q$ is the quotient functor.
\end{thm}
\begin{rem}\label{R:proj-recoll} In the above examples write $\M_i=(\C_i, \W_i)$, for every $i=1,2,3$.  We have that $\C_i\cap\W_i=\widetilde{\A}\cap\widetilde{\B}$. Moreover, dually to \cite[Proposition 3.2]{G7} $\C_i$ is a Frobenius category with the projective-injective objects being exactly the complexes in $\widetilde{\A}\cap\widetilde{\B}$. Thus $(\widetilde{\A}\cap dw\B)/\sim$ and $(dg{\A}\cap dw\B)/\sim$ are the stable categories and they are also equivalent to the homotopy categories $K(\widetilde{\A}\cap dw \B)$ and $K(dg{\A}\cap dw \B)$. Moreover,  all the three terms in the recollement are equivalent to the homotopy categories of the three model structures on $\Ch(\B)$ corresponding to the projective cotorsion pairs $\M_1, \M_2, \M_3$. Furthermore, $\Ch(\B)/ex\B$ is equivalent to the derived category of $R$ as we will see more explicitly later in Remark~\ref{R:derived-of-R}.
\end{rem}
\vskip 0.5cm
By \cite{Nee90}, the derived category of $\Ch(\B)$ is the quotient of $\Ch(\B)$ modulo the acyclic complexes in  $\Ch(\B)$, that is the complexes in $\widetilde{\B}$. Thus we need an exact model structure on $\Ch(\B)$ with $\widetilde{\B}$ as the class of trivial objects. This is provided by Example~\ref{E:examples-proj}~(1).
\begin{thm}\label{T:derived-B} In the setting of Example~\ref{E:examples-proj}~(1), $\M_1=\Big(dg{\A}\cap dw \B, \widetilde{\B}, dw\B\Big)$ is an exact model structure in the category $\Ch(\B)$.
 In particular, we can define the derived category $\D(\B)$ as the quotient  $\Ch(\B)/\widetilde{\B}$.

 Moreover, we have the following triangle equivalences between the derived category of $\Ch(\B)$ and the homotopy category of the model structure $\M_1$:
 \[\D(\B)\cong\ho(\M_1)\cong \dfrac{dg{\A}\cap dw \B}{\widetilde{\A}\cap\widetilde{\B}}.\]
 \end{thm}
 \begin{proof} The projective cotorsion pair $\Big(dg{\A}\cap dw \B, \widetilde{\B}\Big)$ in $\Ch(\B)$ of  Example~\ref{E:examples-proj}~(1) corresponds to  the exact model structure
 $\Big(dg{\A}\cap dw \B, \widetilde{\B}, dw \B\Big)$.
 \end{proof}
 Another way to obtain the exact model structure of Theorem~\ref{T:derived-B} is to use results by Gillespie in \cite{G4}, \cite{G8} and \cite{G9}.

 \begin{thm}\label{T:Gill-B} Let $(\A, \B)$ be a complete hereditary cotorsion pair generated by a set of objects
in a Grothendieck category $\G$. The two complete hereditary cotorsion pairs
 $({}^\perp{} dw\B, dw\B)$ and $(dg\A, \tilde\B)$ in $\Ch(\G)$  give rise to a cofibrantly generated model structure $\M=\Big(dg\A, \V, dw\B\Big)$ in $\Ch(\G )$ satisfying $\V\cap dw\B=\tilde\B$ and $\V\cap dg\A={}^\perp{} dw\B$ whose restriction in $\Ch(\B)$ is the exact model structure $\M_1=\Big(dg{\A}\cap dw \B, \widetilde{\B}, dw\B\Big)$ of Theorem~\ref{T:derived-B}.

 Moreover, if $(\A, \B)$ is generated by a set of finitely presented objects then the model structure $\M=\Big(dg\A, \V, dw\B\Big)$ in $\Ch(R)$ is finitely generated hence its homotopy category is compactly generated.

\end{thm}

 \begin{proof} By  \cite[Proposition 4.3 and Proposition 4.4]{G4} the cotorsion pairs $({}^\perp{} dw\B, dw\B)$ and $(dg\A, \tilde\B)$ are small and they are hereditary since $(\A, \B)$ is hereditary. The existence of the model structure $\M$ in $ \Ch(\G )$ follows by \cite[Theorem 1.1]{G9}. The fact that the model structure is cofibrantly generated follows by \cite[Section 7.4]{Hov02}.
 The last statement follows also by  \cite[Section 7.4]{Hov02}.
\end{proof}
  Combining Theorem~\ref{T:derived-B} with Theorem~\ref{T:Gill-B} we obtain the following consequence:
 \begin{cor}\label{C:compactly generated} Let $(\A, \B)$ be a hereditary cotorsion pair in $\Modr R$ generated by a set of finitely presented modules. Then the derived category $\D(\B)\cong\Ch(\B)/\tilde\B$ is compactly generated. In particular, if $R$ is a coherent ring and $(\A,\FpInj)$ is the complete cotorsion pair generated by all finitely presented modules, then $\D(\FpInj)$ is compactly generated.

 \end{cor}
\begin{proof} Only the second statement needs a comment. If $R$ is a coherent ring, then the complete cotorsion pair $(\A, \FpInj)$ is hereditary. \end{proof}

 \section{Injective cotorsion pairs in the exact category $\Ch(\B)$}\label{S:B-2}

 We look now for injective cotorsion pairs in $\Ch(\B)$ in order to get localizing cotorsion triples in $\Ch(\B)_{dw}$.

First we state the dual of \cite[Proposition 7.2]{G7}.
\begin{prop}\label{P:prop7.2} Let $\G$ be a Grothendieck category and let $(\W, \I)$ be an injective cotorsion pair in $\Ch(\G)$ with $\I\subseteq dw \Inj$ and let $(\A, \B)$ be a complete hereditary cotorsion pair in $\G$. Then \[\Big(\W\cap dw \B, \I\Big)\] is an injective cotorsion pair in $\Ch(\B)$ and \[\Big(\W\cap dw \B, [\I]_K\Big)\] is a localizing cotorsion pair in the Frobenius category $\Ch(\B)_{dw}$, where a complex $X\in \Ch(\B)$ belongs to $[\I]_K$ if and only if it is chain homotopy equivalent to a complex in $\I$.
\end{prop}

\begin{proof} The first statement follows by Proposition~\ref{P:inducing}. The second statement follows by \cite[Theorem 6.3]{G7}.
\end{proof}

We exhibit three examples of injective cotorsion pairs in $\Ch(\G)$ satisfying the assumptions of Proposition~\ref{P:prop7.2}.

When needed, we assume some extra conditions on $\G$ like in (2) below.
 \begin{expl}\label{E:examples-inj} Let $(\A, \B)$ be a complete hereditary cotorsion pair in a Grothendieck  category $\G$.
 \begin{enumerate}
\item  By Proposition~\ref{P:complete?}~(2) we have that $({}^\perp{} dw \Inj, dw \Inj)$ is an injective cotorsion pair  in $\Ch(\G)$ (notice that $(\G, \Inj)$ is generated by a set).
 Hence by Proposition~\ref{P:prop7.2} and \cite[Theorem 6.3]{G7} we obtain the injective cotorsion pair in $\Ch(\B)$:
  \[\N_1=\Big({}^\perp{} dw \Inj\cap dw \B, dw \Inj\Big)\] and the localizing cotorsion pair $\Big({}^\perp{} dw \Inj\cap dw \B, [dw \Inj]_K\Big)$ in $\Ch(\B)_{dw}.$
 \item If $\G$ has a generator of finite projective dimension, by the same references as in part (1) $({}^\perp{} ex \Inj, ex \Inj)$ is an injective cotorsion pair  in $\Ch(\G)$ giving rise to the injective cotorsion pair in $\Ch(\B)$:
 \[\N_2=\Big({}^\perp{} ex \Inj\cap dw \B, ex \Inj\Big)\] and to the localizing cotorsion pair $\Big({}^\perp{} ex \Inj\cap dw \B, [ex \Inj]_K\Big)$ in $\Ch(\B)_{dw}.$
 \item  By Proposition~\ref{P:description-tilde}~(3) $(\E, dg\Inj)$ is a complete hereditary cotorsion pair in $\Ch(\G)$; again by the same references as in part (1), in  $\Ch(\B)$ we obtain the injective cotorsion pair
 \[\N_3=\Big(ex\B,dg \Inj\Big)\] and the localizing cotorsion pair $\Big(ex\B, [dg \Inj]_K\Big)$ in $\Ch(\B)_{dw}.$

\end{enumerate}
\end{expl}

 In the above examples we  write $\N_i=(\W_i, \R_i)$, for every $i=1,2,3$. Then, $\R_2\subseteq \W_3$ and
$\W_2\cap\W_3=\W_1$. In fact, by the analogous of \cite[Theorem 4.7]{G4} in a Grothendieck category,  ${}^\perp{} ex \Inj\cap \E={}^\perp{} dw\Inj$. Thus, they satisfy the conditions of  \cite[Theorem 3.4]{G7} and allow to build a recollement which, as we will point out in Remark~\ref{R:inj-recoll}, is nothing else than Krause's recollement~\cite{Kr05}.

\begin{thm}\label{T:recollement-inj} Let $\G$ be a Grothendieck category with a generator of finite projective dimension and let $(\A, \B)$ be a complete hereditary cotorsion pair in $\G$.The three injective cotorsion pairs in Example~\ref{E:examples-inj} satisfy the conditions of \cite[Theorem 3.4]{G7}, so that we get the recollement:

\[
\xymatrix{\dfrac{ex \Inj}{\sim} \ar[rr]^{inc} &&\dfrac{dw\Inj}{\sim} \ar@/^2pc/
[ll]\ar@/_2pc/[ll] \ar[rr]^{Q}
&&{\Ch(\B)/ex\B}\ar@/^2pc/ [ll] \ar@/_2pc/ [ll] }
\]
\vskip 0.7cm

where $\sim$ denotes the homotopy relation associated to the corresponding model structure and coincides with the chain homotopy relation; moreover, $inc$ is the inclusion and $Q$ is the quotient functor.
\end{thm}

\begin{rem}\label{R:inj-recoll}
\begin{enumerate}
{\,}
\item Writing $\N_i=(\W_i, \R_i)$, for every $i=1,2,3$, \cite[Proposition 3.2]{G7} implies that $\R_i$ is a Frobenius category with the projective-injective object being exactly  the injective objects in $\Ch(\G)$ or, equivalently, in $\Ch(\B)$. Note that, for every $i=1,2,3$,  $\R_i\cap\W_i$ is the class of injective objects in $\Ch(\B)$. Thus, $\R_1/\sim$ is equivalent to the  homotopy category $K(\Inj)$ of the complexes with injective terms and $\R_2 /\sim$ is equivalent to $K(ex\Inj)$ the full subcategory of $K(\Inj)$ consisting of exact complexes of injectives. Moreover, $\Ch(\B)/ex\B$ is equivalent to the derived category $\D(R)$, as it will be clear from Theorem~\ref{T:triple-B}.

 That is, Theorem~\ref{T:recollement-inj} is yet another instance of Krause's recollement \cite{Kr05}, which was recovered also in \cite{Beck14}.

\item The complexes in ${}^\perp{} dw \Inj$ are called coacyclic in \cite{Pos} (see also \cite{Stopurity} and \cite{Beck14}). By \cite[Proposition 6.9]{Stopurity} the homotopy category of the injective cotorsion pair $\N_1$ is equivalent to $K(\Inj)$ and called the coderived category of $\G$. Thus the central term of the above recollement is equivalent to the coderived category of $\G$.
\end{enumerate}
\end{rem}
Combining some of the previous examples we can state the following:
\begin{thm}\label{T:triple-B}  Let $(\A, \B)$ be a complete hereditary cotorsion pair in a  Grothendieck category $\G$ such that $\A$ contains a generator of finite projective dimension.
The triple $\Big([{}^\perp{} ex\B\cap dw \B]_K, ex\B, [dg \Inj]_K\Big)$ is a localizing cotorsion triple in $\Ch(\B)_{dw}$. If  $\X=[{}^\perp{} ex\B\cap dw \B]_K$, $\Y=ex\B$, $\Z=[dg \Inj]_K$ then there are equivalences of triangulated categories:
\[\frac{[{}^\perp{} ex\B\cap dw \B]_K}{\sim} \cong \frac{\Ch(\B)}{ex\B} \cong  \frac{[dg \Inj]_K}{\sim}\]
where $\sim$ is the chain homotopy equivalence
and a recollement:
\vskip0.7cm
\[
\xymatrix{\dfrac{ex\B}{\sim} \ar[rr]^{inc} &&{K(\B)} \ar@/^2pc/[ll]\ar@/_2pc/[ll] \ar[rr]^{Q}
&&{\dfrac{\Ch(\B)}{ex\B }}\ar@/^2pc/ [ll] \ar@/_2pc/ [ll] }
\]
\vskip0.7cm
where the middle term is  the homotopy category $K(\B)$ of the complexes with terms in $\B$ modulo the chain homotopy equivalence.

\end{thm}
\begin{proof} By Example~\ref{E:examples-proj}~(3)  and Example~\ref{E:examples-inj}~(3), we have two localizing cotorsion pairs $\Big([{}^\perp{} ex\B\cap dw \B]_K, ex\B\Big)$ and $\Big(ex\B,  [dg \Inj]_K\Big)$ in $\Ch(\B)_{dw}$. Let $\X=[{}^\perp{} ex\B\cap dw \B]_K$, $\Y=ex\B$ and $\Z=[dg \Inj]_K$, then $(\X, \Y, \Z)$ is a localizing cotorsion triple in $\Ch(\B)_{dw}$. Moreover, we can consider also the localizing cotorsion pair $(\Ch(\B)_{dw}, \W)$ where $\W$ is the class of projective-injective (contractible) objects in $\Ch(\B)_{dw}$ so that we have three localizing cotorsion pairs in $\Ch(\B)_{dw}$:
\[\clL_1=(\Ch(\B)_{dw}, \W), \clL_2=(\X, \Y), \clL_3=(\Y, \Z)\]
which satisfy the conditions in Theorem~\ref{T:recollement-proj} hence they give rise to a recollement where the central term is the stable category $\Ch(\B)_{dw}/\sim$. Since a map in $\Ch(\B)$ is null homotopic if and only if it factors through a contractible complex, we have that the stable category of $\Ch(\B)_{dw}$ is  the homotopy category $K(\B)$ of the complexes with terms in $\B$ modulo the chain homotopy equivalence.

Then the conclusion follows by  \cite[Corollary 4.5]{G7} .
\end{proof}

\begin{rem}\label{R:derived-of-R} From the equivalence $\dfrac{\Ch(\B)}{ex\B} \cong \dfrac{ [dg \Inj]_K}{\sim}$ we see that  $\dfrac{\Ch(\B)}{ex\B}$  is equivalent to the usual derived category $\D(\G)$.
\end{rem}

\section{When is $\tilde{\B}$ the central term of a localizing cotorsion triple in $\Ch(\B)_{dw}$?}\label{S:tildeB}

In Example~\ref{E:examples-inj}~(3) we have shown that there is an injective cotorsion pair $\Ch(\B)$ with $ex\B$
 as left term and Example~\ref{E:examples-proj}~(1) provides a projective cotorsion pair in $\Ch(\B)$ with right component $ \widetilde{\B}$.

 Our aim will be to find cotorsion pairs $(\A, \B)$ for which there exist an injective cotorsion pair $(\widetilde{\B}, \R)$ in $\Ch(\B)$ with $
\R\subseteq dw\Inj$ in order to obtain a localizing cotorsion triple in $\Ch(\B)_{dw}$ with $\widetilde{\B}$ as central term.

 Section~\ref{S:2} provides examples of cotorsion pairs $(\A, \B)$ such that $ex\B=\widetilde{\B}$.

A first case appears in Proposition~\ref{P:finite-proj-dim}.
\begin{prop}\label{P:new-triple}
Let $(\A, \B)$ be a complete hereditary cotorsion pair in $\G$ with $\A\subseteq \clP$ ($\clP$ the class of objects with finite projective dimension). Then $ex\B=\tilde\B$, hence $(\tilde\B, dg\Inj)$ is an injective cotorsion pair in $\Ch(\B)$ and there is a recollement as in Theorem~\ref{T:triple-B} with $ex \B$ replaced by $\tilde\B$.

In particular, the derived category $\D(\B)$ of $\B$ is equivalent to the usual derived category of $\G$.
\end{prop}

\begin{cor}\label{C:recoll-tilting} Let $(\A, \T)$ be an $n$-tilting cotorsion pair in $\Modr R$. For the tilting class $\T$ we have a recollement:
\vskip0.7cm
\[
\xymatrix{\dfrac{\widetilde{\T}}{\sim} \ar[rr]^{inc} &&{K(\T)} \ar@/^2pc/[ll]\ar@/_2pc/[ll] \ar[rr]^{Q}
&&{\dfrac{\Ch(\T)}{\widetilde{\T} }}\ar@/^2pc/ [ll] \ar@/_2pc/ [ll] }
.\]
\vskip0.7cm
\end{cor}

\begin{proof} By Proposition~\ref{P:tilt-cotil} $ex\T=\widetilde{\T}$, hence the conclusion follows by Proposition~\ref{P:new-triple}.
\end{proof}

 \begin{prop}\label{P:ex-Cot} The cotorsion pair $(\Flat, \Cot)$ in $\Modr R$ satisfies $ex\Cot=\widetilde{\Cot}$,  hence it induces a recollement:
\vskip0.7cm
\[
\xymatrix{\dfrac{\widetilde{\Cot}}{\sim} \ar[rr]^{inc} &&{K(\Cot)} \ar@/^2pc/[ll]\ar@/_2pc/[ll] \ar[rr]^{Q}
&&{\dfrac{\Ch(\Cot)}{\widetilde{\Cot} }}\ar@/^2pc/ [ll] \ar@/_2pc/ [ll] }
.\]
\vskip0.7cm
\end{prop}
\begin{proof} The fact that $ex\Cot=\widetilde{\Cot}$ in $\Ch(R)$ is proved in \cite[Theorem 4.1~(2)]{BCIE}.
  Hence, the conclusion follows  by Theorem~\ref{T:triple-B}.\end{proof}

 \begin{prop}\label{P:B-infty} The cotorsion pair $(\A^{\infty}, \B_{\infty})$ from Notation~
 \ref{N:notations}~(5) satisfies $ex\B_{\infty}=\widetilde{\B_{\infty}}$,  hence it induces a recollement:
 \vskip0.7cm
\[
\xymatrix{\dfrac{\widetilde{\B_{\infty}}}{\sim} \ar[rr]^{inc} &&{K(\B_{\infty})} \ar@/^2pc/[ll]\ar@/_2pc/[ll] \ar[rr]^{Q}
&&{\dfrac{\Ch(\B_{\infty})}{\widetilde{\B_{\infty}} }}\ar@/^2pc/ [ll] \ar@/_2pc/ [ll] }
.\]
\vskip0.7cm
\end{prop}
\begin{proof} By Proposition~\ref{P:ex-B-infty} we have $ex\B_{\infty}=\widetilde{\B_{\infty}}$, hence the conclusion follows again by Theorem~\ref{T:triple-B}.\end{proof}

In view of Lemma~\ref{L:Cot-inj} we have the following characterization.

\begin{prop}\label{P:no-tilde} Let $(\A, \B)$ be a complete hereditary cotorsion pair in $\Modr R$ with $\B\supseteq  \B_{\infty}$.

Then in $\Ch(\B)$ there exists an injective cotorsion pair $(\widetilde{\B}, \R)$ with $
\R\subseteq dw\Inj$ \iff $ex\B=\tilde\B$.\end{prop}

\begin{proof}
Assume that in $\Ch(\B)$ there is an injective cotorsion pair $(\tilde \B, \R) $ with $
\R\subseteq dw\Inj$.  This means that $\R=\tilde\B{}^\perp{} \cap dw \B$. By Lemma~\ref{L:Cot-inj}~(3) and (4), $\R= dg\Inj$ and ${}^\perp{} \R\cap dw\B=ex\B$. Then, $ex\B=\tilde\B$.

Conversely, if $ex\B=\tilde\B$ then $(\tilde\B, dg\Inj)$ is an injective cotorsion pair in $\Ch(\B)$ by Example~\ref{E:examples-inj}~(3).
\end{proof}
\begin{rem}\label{R:many-B} Note that complete hereditary cotorsion pairs $(\A, \B)$ satisfying $\B\supseteq  \B_{\infty}$ may be abundant, since  $ \B_{\infty}$ may be rather small.

For instance, if the little finitistic dimension of $R$  is finite (e.g. $R$ is  semihereditary), then $ \B_{\infty}$ coincides with the class of injective modules (see Proposition~\ref{P:semihereditary}).
\end{rem}

A positive answer to the question in the title of this section is provided by \sto\  in \cite{Stopurity} for  the cotorsion pair $(\A, \FpInj)$ generated by the class of finitely presented modules over a coherent ring $R$.
In view of Example~\ref{E:examples-proj}~(1) and Example~\ref{E:examples-inj}~(1), we restate \Sto's theorem in our notations.
 \begin{prop}\label{P:fp-inj}(\cite[Proposition 6.11, Theorem 6.12]{Stopurity})  Let $R$ be a coherent ring and let $(\A, \FpInj)$ be the complete hereditary cotorsion pair generated by the class of finitely presented modules. Then:
 \[{}^\perp{} dw\Inj\cap dw \FpInj=\widetilde{\FpInj}\]
  hence $\Big([dg\A\cap dw\FpInj]_K, \widetilde{\FpInj}, [dw \Inj]_K\Big)$ is a localizing cotorsion triple in $\Ch(\FpInj)_{dw}$. There are equivalences:
\[\frac{[dg\A\cap dw\FpInj]_K}{\sim} \cong \frac{\Ch(\FpInj)}{ \widetilde{\FpInj}} \cong  \frac{[dw \Inj]_K}{\sim}\]
where $\sim$ is the chain homotopy equivalence
and a recollement:
\vskip0.7cm
\[
\xymatrix{\dfrac{ \widetilde{\FpInj}}{\sim} \ar[rr]^{inc} &&{K(\FpInj)} \ar@/^2pc/[ll]\ar@/_2pc/[ll] \ar[rr]^{Q}
&&{\dfrac{\Ch(\FpInj)}{ \widetilde{\FpInj}}\cong \D(\FpInj)}\ar@/^2pc/ [ll] \ar@/_2pc/ [ll] }
.\]
\vskip0.7cm
\end{prop}

We exhibit now another case of cotorsion pairs giving rise to a result analogous to Proposition~\ref{P:fp-inj}

\begin{prop}\label{P:B-inj-n} Let $R$ be a coherent ring and let $(\A, \B)$ be a complete hereditary cotorsion pair in $\Modr R$. Assume that $\B\subseteq \FpInj$ and that $\B\subseteq \I_n$ (the class of modules of injective dimension at most $n$).

 Then, every Fp-injective  $\B$-periodic module belongs to $\B$. Thus
 $\widetilde{\FpInj}\cap dw \B=\tilde{\B}$, $ ^\perp dw\Inj \cap dw\B=\tilde{\B}$ and we have a recollement:
\vskip0.7cm
\[
\xymatrix{\dfrac{ \widetilde{\B}}{\sim} \ar[rr]^{inc} &&{K(\B)} \ar@/^2pc/[ll]\ar@/_2pc/[ll] \ar[rr]^{Q}
&&{\dfrac{\Ch(\B)}{ \widetilde{\B}}\cong \D(\B)}\ar@/^2pc/ [ll] \ar@/_2pc/ [ll] }
.\]
\vskip0.7cm

\end{prop}
 \begin{proof} Let $(\ast)\quad 0\to M\to B\to M\to 0$ be an exact sequence with $M$ Fp-injective and $B\in \B$. Let $ 0\to M\to E\to M_1\to 0$ be an exact sequence with $E$ injective; then, $M_1$ is Fp-injective. An application of the horseshoe lemma gives the following commutative diagram:
 \[\xymatrix{
&0 \ar[d]  & 0 \ar[d] & 0 \ar[d] \\
0 \ar[r] & M \ar[d] \ar[r] &B\ar[r] \ar[d] & M\ar[d]  \ar[r] & 0 \\
0 \ar[r] & E \ar[d] \ar[r] &E\oplus E\ar[r] \ar[d] & E\ar[d]  \ar[r] & 0  \\
0 \ar[r] & M_1 \ar[d]\ar[r] &D\ar[r] \ar[d] & M_1\ar[d]  \ar[r] & 0 \\
&0 & 0 & 0
}
\]
where  $D\in \B$, since $\B$ is coresolving. We have inj.dim $D=$ inj.dim $B-1$, hence w.l.o.g. we can assume that in our starting sequence  $(\ast)$ $B$ has injective dimension at most $1$. Thus, in the above diagram we have that $D$ is injective and, by Fact~\ref{F:periodic}~(1), we conclude that $M_1$ is injective. The latter implies that inj.dim $M\leq 1$.
Let $A\in \A.$ Then $0=\Ext^1_R(A, B)\to \Ext^1_R(A, M)\to \Ext^2_R(A, M)=0$, hence $M\in \B$ and
$\widetilde{\FpInj}\cap dw \B=\tilde{\B}$  by Fact~\ref{F:periodic}~(2).  Hence, the equality $ ^\perp dw\Inj \cap dw\B=\tilde{\B}$ is obtained by intersecting with $dw\B$ the equality $^\perp{} dw\Inj\cap dw \FpInj=\widetilde{\FpInj}$ from \cite[Proposition 6.11 ]{Stopurity}.

The existence of a recollement as in the statement follows by the same arguments as in the proof of Proposition~\ref{P:fp-inj} applied to the cotorsion pair $(\A, \B)$ in the assumptions.\end{proof}

 \section{Injective cotorsion pairs in the exact category $\Ch(\A)$}

 In this section we state results dual to the ones in Section~\ref{S:B}. Their proofs are obtained  by dual arguments.

Starting with a complete hereditary cotorsion pair $(\A, \B)$ in a Grothendieck  category $\G$, we exhibit three examples of injective cotorsion pairs in $\Ch(\G)$ satisfying the assumptions of Proposition~\ref{P:prop7.3}. Note that the examples below satisfy Proposition~\ref{P:prop7.3-dual} since $\tilde{\A}$, $dg\A$ and $ex\A$  are thick in $\Ch(\A)$ by Lemma~\ref{L:HOM} and they clearly contain the contractible complexes with terms in $\A$.

 \begin{expl}\label{E:exam-A-inj}
 \begin{enumerate}
\item The complete hereditary cotorsion pair $( \widetilde{\A}, dg\B)$ in $\Ch(\G)$ satisfies the conditions in  Proposition~\ref{P:prop7.3}, hence we have the injective cotorsion pair:
\[\Delta_1=\Big(\tilde\A, dg{\B}\cap dw\A\Big)\] in $\Ch(\A)$ and the localizing cotorsion pair $\Big(\tilde\A,  [dg{\B}\cap dw\A]_K\Big)$ in $\Ch(\A)_{dw}.$

\item Also the complete hereditary cotorsion pair $(dg\A, \widetilde{\B})$ in $\Ch(\G)$ satisfies the conditions in  Proposition~\ref{P:prop7.3},  hence we have the injective cotorsion pair:
 \[\Delta_2=\Big(dg\A, \tilde\B\cap dw\A\Big)\] in $\Ch(\A)$ and the localizing cotorsion pair $\Big(dg\A,  [\tilde\B\cap dw\A]_K\Big)$ in $\Ch(\A)_{dw}.$

\item If $\A$ is deconstructible, then by Proposition~\ref{P:complete?}~(6) $(ex\A, ex\A{}^\perp{})$ is a complete hereditary cotorsion pair in $\Ch(\G)$ and it satisfies the conditions in Proposition~\ref{P:prop7.3}, hence we have the injective cotorsion pair
 \[\Delta_3=\Big(ex\A, ex\A{}^\perp{}\cap dw\A\Big)\] in $\Ch(\A)$ and the localizing cotorsion pair $\Big(ex\A,  [ex{}^\perp{}\A\cap dw\B]_K\Big)$ in $\Ch(\A)_{dw}.$

 \end{enumerate}
 \end{expl}
The three injective cotorsion pairs $\Delta_1, \Delta_2, \Delta_3$ of Example~\ref{E:exam-A-inj} satisfy the conditions of \cite[Theorem 3.4]{G7}, hence we have:
\begin{thm}\label{T:recollement-inj-A} Let $(\A, \B)$ be a complete hereditary cotorsion pair in a  Grothendieck category $\G$ such that $\A$ is deconstructible. Then, there is a recollement:

\vskip 0.7cm
\[
\xymatrix{\dfrac{\widetilde{\B}\cap dw \A}{\sim} \ar[rr]^{inc} &&\dfrac{dg{\B}\cap dw \A}{\sim} \ar@/^2pc/
[ll]\ar@/_2pc/[ll] \ar[rr]^{Q}
&&{\Ch(\A)/ex\A}\ar@/^2pc/ [ll] \ar@/_2pc/ [ll] }
\]
\vskip 0.7cm

where $\sim$ denotes the homotopy relation associated to the corresponding model structure and coincides with the chain homotopy relation; moreover, $inc$ is the inclusion and $Q$ is the quotient functor.
\end{thm}
\begin{thm}\label{T:derived-A}  In the setting of Example~\ref{E:exam-A-inj}~(1), $\Delta_1=\Big(dw{\A}, \widetilde{\A}, dg\B\cap dw \A\Big)$ is an exact model structure in the category $\Ch(\A)$.
 In particular, we can define the derived category $\D(\A)$ as the quotient  $\Ch(\A)/\widetilde{\A}$.

 Moreover, we have the following triangle equivalences between the derived category of $\D(\A)$ and the homotopy category of the model structure $\Delta_1$:
 \[\D(\A)\cong\ho(\Delta_1)\cong \dfrac{dg{\B}\cap dw \A}{\widetilde{\A}\cap\widetilde{\B}}.\]
 \end{thm}
 \begin{proof} By Example~\ref{E:exam-A-inj}~(1) in $\Ch(\A)$ we have the injective cotorsion pair $\Big(dg{\A}\cap dw \B, \widetilde{\B}\Big)$, hence the exact model structure
 $\Big(dg{\A}\cap dw \B, \widetilde{\B}, dw \B\Big)$.
 \end{proof}
 Another way to obtain the exact model structure of Theorem~\ref{T:derived-A} is to use results by Gillespie in \cite{G4}, \cite{G8} and \cite{G9}.

 \begin{thm}\label{T:Gill-A}  Let $(\A, \B)$ be a complete hereditary cotorsion pair in $\G$ such that $\A$ is deconstructible.  The two cotorsion pairs
 $(dw\A, dw\A{}^\perp{} )$ and $(\tilde\A, dg\B)$ in $\Ch(\G)$ are hereditary and complete and give rise to a cofibrantly generated model structure $\N=\Big(dw\A, \W, dg\B\Big)$ in $\Ch(\G)$ satisfying $\W\cap dw\A=\tilde\A$ and $\W\cap dg\B=dw\A{}^\perp{} $ whose restriction in $\Ch(\A)$ is the exact model structure $\N_1=\Big(dw{\A}, \widetilde{\A},  dw \A\cap dg\B\Big)$ of Theorem~\ref{T:derived-A}.

\end{thm}

 \begin{proof} The smallness of the cotorsion pairs $(dw\A, dw\A{}^\perp{} )$ and $(\tilde\A, dg\B)$ follow by the fact that $dw\A$ and $\tilde\A$ are deconstructible in $\Ch(\G)$ (see  Proposition~\ref{P:complete?}~(3) and ~\ref{P:description-tilde}~(3))  and they are hereditary since $(\A, \B)$ is hereditary. The existence of the model structure $\N$ in $ \Ch(R)$ follows by \cite[Theorem 1.1]{G9}. The fact that the model structure is cofibrantly generated follows by \cite[Section 7.4]{Hov02}.
\end{proof}
  \section{Projective cotorsion pairs in the exact category $\Ch(\A)$}
   In this section we state results dual to the ones in Section~\ref{S:B-2}.
 First we state the analogous of \cite[Proposition 7.2 ]{G7}.

 \begin{prop}\label{P:7.2-Groth} Let $\G$ be a Grothendieck category with enough projective objects and let $(\clP, \W)$ be a projective cotorsion pair in $\Ch(\G)$ with $\clP\subseteq dw\Proj$.
 Let $(\A, \B)$ be a complete hereditary cotorsion pair in $\G$. Then,
 \[(\clP, \W\cap dw\A)\] is a projective cotorsion pair in $\Ch(\A)$ and \[\Big([\clP]_K, \W\cap dw \A\Big)\] is a localizing cotorsion pair in the Frobenius category $\Ch(\A)_{dw}$. A complex $X\in \Ch(\A)$ is in $[\clP]_K$ if and only if it is chain homotopy equivalent to a complex in $P\in \clP$.
 \end{prop}
 \begin{proof} $(\clP, \W\cap dw\A)$ is a complete cotorsion pair by Proposition~\ref{P:inducing} and it is automatically a projective cotorsion pair. The second statement follows by the dual of \cite[Theorem 6.3, Proposition 6.4]{G7}.
\end{proof}
\begin{expl}\label{E:exam-A-proj}

Starting with a complete hereditary cotorsion pair $(\A, \B)$ in a Grothendieck  category $\G$ with enough projective objects, we exhibit three examples of projective cotorsion pairs in $\Ch(\G)$ satisfying the assumptions of Proposition~\ref{P:7.2-Groth}.

 \begin{enumerate}
\item  By Proposition~\ref{P:complete?}~(3), $(dw\Proj, dw \Proj{}^\perp{})$ is a complete cotorsion pair  in $\Ch(\G)$, and it is a projective cotorsion pair. By Proposition~\ref{P:7.2-Groth},  we have the projective cotorsion pair:
 \[\Gamma_1=\Big(dw\Proj, dw \Proj{}^\perp{}\cap dw\A\Big)\] in $\Ch(\A)$ and the localizing cotorsion pair $\Big([dw\Proj]_K, dw \Proj{}^\perp{}\cap\A\Big)$ in $\Ch(\A)_{dw}.$
 \item By Proposition~\ref{P:complete?}~(6), $(ex\Proj, ex\Proj{}^\perp{})$ is a projective cotorsion pair in $\Ch(\G)$ and by Proposition~\ref{P:7.2-Groth} we have the projective cotorsion pair:
  \[\Gamma_2=\Big(ex\Proj, ex \Proj{}^\perp{}\cap dw\A\Big)\] in $\Ch(\A)$ and the localizing cotorsion pair $\Big([ex\Proj]_K, ex \Proj{}^\perp{}\cap\A\Big)$ in $\Ch(\A)_{dw}.$
  \item Since $(dg\Proj, \E )$ is a projective cotorsion pair in $\Ch(\G)$, by Proposition~\ref{P:7.2-Groth} we have the projective cotorsion pair:
  \[\Gamma_3=\Big(dg\Proj, ex\A\Big)\]  in $\Ch(\A)$ and the localizing cotorsion pair $\Big([dg\Proj]_K, ex\A\Big)$ in $\Ch(\A)_{dw}.$
\end{enumerate}
\end{expl}

 The above three examples $\Gamma_1, \Gamma_2, \Gamma_3$ of projective cotorsion pairs  in $\Ch(\A)$ satisfy the assumptions of \cite[Theorem 3.5]{G7}. Hence we obtain:
  \begin{thm}\label{T:recollement-proj-A} If $\G$ is a Grothendieck category with enough projective objects and $(\A, \B)$ is a complete hereditary cotorsion pair in $\G$, there is a recollement %
\vskip 0.7cm
\[
\xymatrix{\dfrac{ex \Proj}{\sim} \ar[rr]^{inc} &&\dfrac{dw\Proj}{\sim} \ar@/^2pc/
[ll]\ar@/_2pc/[ll] \ar[rr]^{Q}
&&{\Ch(\A)/ex\A}\ar@/^2pc/ [ll] \ar@/_2pc/ [ll] }
\]
\vskip 0.7cm

where $\sim$ denotes the homotopy relation associated to the corresponding model structure and coincides also with the chain homotopy relation; moreover, $inc$ is the inclusion, $Q$ is the quotient functor.  In particular, the central term is the chain homotopy category $K(\Proj)$ of the complexes with projective components and the right hand term is equivalent to the derived category of $\G$.
\end{thm}
Moreover, we have:

\begin{thm}\label{T:triple-A} In the assumptions of Theorem~\ref{T:recollement-proj-A}, the triple
\[\Big([dg\Proj]_K, ex\A, [ex\A{}^\perp{}\cap dw\A]_K\Big)\] is a localizing cotorsion triple in $\Ch(\A)_{dw}$ and there are equivalences of triangulated categories:
\[\frac{[dg\Proj]_K}{\sim} \cong \frac{\Ch(\A)}{ex\A} \cong  \frac{[ex\A{}^\perp{}\cap dw\A]_K}{\sim}\]
where $\sim$ is the chain homotopy equivalence
and a recollement:
\vskip0.7cm
\[
\xymatrix{\dfrac{ex\A}{\sim} \ar[rr]^{inc} &&{K(\A)} \ar@/^2pc/[ll]\ar@/_2pc/[ll] \ar[rr]^{Q}
&&{\dfrac{\Ch(\A)}{ex\A }}\ar@/^2pc/ [ll] \ar@/_2pc/ [ll] }
.\]
\vskip0.7cm
\end{thm}
\begin{proof} By Examples~\ref{E:exam-A-inj}~(3) and Examples~\ref{E:exam-A-proj}~(3)  we have two localizing cotorsion pairs $\Big([dg\Proj]_K, ex\A\Big)$ and $\Big(ex\A, [ex\A{}^\perp{}\cap dw\A]_K)$ in $\Ch(\A)_{dw}$.  Thus, the statement follows by arguing as in the proof of Theorem~\ref{T:triple-B}.
\end{proof}

\section{When is $\tilde{\A}$ the central term of a localizing cotorsion triple in $\Ch(\A)_{dw}$?}

By Example~\ref{E:exam-A-inj}~(1) we have shown that in $\Ch(\A)$ there is an injective cotorsion pair  with $\tilde \A$
 as left term and Example~\ref{E:exam-A-proj}~(3) provides a projective cotorsion pair in $\Ch(\A)$ with right component $ex{\A}$.
 Our aim will be to find cotorsion pairs $(\A, \B)$ for which there exist a projective cotorsion pair $(\C, \widetilde{\A})$ in $\Ch(\A)$ with $
\C\subseteq dw\Proj$ in order to obtain a localizing cotorsion triple in $\Ch(\A)_{dw}$ with $\widetilde{\A}$ as central term.

The most famous example of this situation is provided by the cotorsion pair $(\Flat, \Cot)$. In fact, by \cite[Theorem 8.6]{Nee08} $dw(\Proj)^\perp\cap dw{\Flat}=\widetilde{\Flat}$. Hence, as noted by Gillespie in \cite{G7}, Example~\ref{E:exam-A-proj}~(1) provides the wanted example and induces Neeman's recollement, that is the recollement as in Theorem~\ref{T:triple-A}:
\vskip0.7cm
\[
\xymatrix{(a)\quad\dfrac{\widetilde{\Flat} }{\sim} \ar[rr]^{inc} &&{K(\Flat)} \ar@/^2pc/[ll]\ar@/_2pc/[ll] \ar[rr]^{Q}
&&{\dfrac{\Ch(\Flat)}{\widetilde{\Flat} }}\ar@/^2pc/ [ll] \ar@/_2pc/ [ll] }
.\]
\vskip0.7cm

Another case of cotorsion pairs giving rise to a result analogous to the previous one is provided by the following:
\begin{prop}\label{P:A-flat} Let $(\A, \B)$ be a complete hereditary cotorsion pair in $\Modr R$. Assume that $\A\subseteq \Flat$ and that $\A\subseteq \clP_n$, where $\clP_n$ is the class of modules of projective dimension at most $n$.
 Then, every flat $\A$-periodic module belongs to $\A$. Thus $\widetilde{\Flat}\cap dw \A=\tilde{\A}$, $ (dw\Proj)^\perp\cap dw\A=\tilde{\A}$ and there is a recollement:
 \vskip0.7cm
\[
\xymatrix{\dfrac{\widetilde{\A} }{\sim} \ar[rr]^{inc} &&{K(\A)} \ar@/^2pc/[ll]\ar@/_2pc/[ll] \ar[rr]^{Q}
&&{\dfrac{\Ch(\A)}{\widetilde{\A} }}\ar@/^2pc/ [ll] \ar@/_2pc/ [ll] }
.\]
\vskip0.7cm

 In particular, if $R$ is commutative, then the above statements apply
 to the class $\A$ of very flat modules. \end{prop}
 \begin{proof} Let $0\to F\to A\to F\to 0$ be an exact sequence with $F$ flat and $A\in \A$. Let $0\to F_1\to P\to F\to 0$ be an exact sequence with $P$ projective; then $F_1$ is flat. As in \cite{Sto16}, an application of the horseshoe lemma gives the following commutative diagram:
 \[\xymatrix{
&0 \ar[d]  & 0 \ar[d] & 0 \ar[d] \\
0 \ar[r] & F_1 \ar[d] \ar[r] &Q\ar[r] \ar[d] & F_1\ar[d]  \ar[r] & 0 \\
0 \ar[r] & P \ar[d] \ar[r] &P\oplus P\ar[r] \ar[d] & P\ar[d]  \ar[r] & 0  \\
0 \ar[r] & F \ar[d]\ar[r] &A\ar[r] \ar[d] & F\ar[d]  \ar[r] & 0 \\
&0 & 0 & 0
}
\]
where  $Q\in \A$, since $\A$ is resolving. We have p.dim $Q=$ p.dim $A-1$, hence w.l.o.g. we can assume that in our starting sequence  $0\to F\to A\to F\to 0$ $A$ has projective dimension at most $1$. Thus, in the above diagram we have that $Q$ is projective and by \cite{BG}, $F_1$ is projective. The latter implies that p.dim $F\leq 1$.
Let $B\in \B.$ Then $0=\Ext^1_R(A, B)\to \Ext^1_R(F, B)\to \Ext^2_R(F, B)=0$, hence $F\in \A$ and
$\widetilde{\Flat}\cap dw \A=\tilde{\A}$ by Fact~\ref{F:periodic}~(2).  Now the equality $ (dw\Proj)^\perp\cap dw\A=\tilde{\A}$ is obtained by intersecting with $dw\A$ the equality $dw(\Proj)^\perp\cap dw{\Flat}=\widetilde{\Flat}$ from \cite[Theorem 8.6]{Nee08}.

The arguments used above to obtain the recollement $(a)$ for the cotorsion pair $(\Flat, \Cot)$ can be repeated for the case of the cotorsion pair $(\A, \B)$ in our assumption to obtain the stated recollement.\end{proof}

Another situation is provided by Proposition~\ref{P:finite-proj-dim}.
\begin{prop}\label{P:new-triple-A}
Let $\G$ be a Grothendieck category with enough projective objects. Let $(\A, \B)$ be a complete hereditary cotorsion pair in $\G$ with $\B\subseteq \I$ ($\I$ the class of objects with finite injective dimension). Then $ex\A=\tilde\A$, hence $(dg\Proj, \tilde{\A})$ is a projective cotorsion pair in $\Ch(\A)$ and there is a recollement as in Theorem~\ref{T:triple-A} with $ex \A$ replaced by $\tilde\A$.
 In particular, the derived category $\D(\A)$ of $\A$ is equivalent to the usual derived category of $\G$.
 \end{prop}
\begin{cor}\label{C:recoll-cotilting} Let $(\C, \B)$ be an $n$-cotilting cotorsion pair in $\Modr R$. For the cotilting class $\C$ we have a recollement:
\vskip0.7cm

\[
\xymatrix{\dfrac{\widetilde{\C}}{\sim} \ar[rr]^{inc} &&{K(\C)} \ar@/^2pc/[ll]\ar@/_2pc/[ll] \ar[rr]^{Q}
&&{\dfrac{\Ch(\C)}{\widetilde{\C} }}\ar@/^2pc/ [ll] \ar@/_2pc/ [ll] }
.\]
\vskip0.7cm
\end{cor}
\begin{proof} By Proposition~\ref{P:tilt-cotil}, $ex\C=\widetilde{\C}$, hence the conclusion follows by Proposition~\ref{P:new-triple-A}.
\end{proof}
\bibliographystyle{alpha}
\bibliography{references}
\end{document}